\documentclass{article}

\usepackage{amsfonts}
\usepackage{amsmath}
\usepackage{amssymb}
\usepackage{amsthm}
\usepackage[left=2.9cm, right=2.9cm, top=2.9cm, bottom=2.9cm]{geometry}
\usepackage{subcaption}
\usepackage{tikz}

\newtheorem{corollary}{Corollary}[section]
\newtheorem{lemma}[corollary]{Lemma}
\newtheorem{proposition}[corollary]{Proposition}
\newtheorem{remark}[corollary]{Remark}
\newtheorem{theorem}[corollary]{Theorem}
\newtheorem{definition}[corollary]{Definition}
\newtheorem{example}[corollary]{Example}

\allowdisplaybreaks

\title{Stability estimates for phase retrieval from discrete Gabor measurements}
\author{Rima Alaifari and Matthias Wellershoff}

\begin{document}
	
	%------------------------------------------- Title --------------------------------------------
	\maketitle
	
	%------------------------------------------ Abstract ------------------------------------------
	\begin{abstract}
		\noindent
		Phase retrieval refers to the problem of recovering some signal (which is often modelled as
		an element of a Hilbert space) from phaseless measurements. It has been shown that in the
		deterministic setting phase retrieval from frame coefficients is always unstable in
		infinite-dimensional Hilbert spaces \cite{Cahill16} and possibly severely ill-conditioned
		in finite-dimensional Hilbert spaces \cite{Cahill16}.
		
		Recently, it has also been shown that phase retrieval from measurements induced by the
		Gabor transform with Gaussian window function is stable under a more relaxed
		\emph{semi-global} phase recovery regime based on \emph{atoll functions}
		\cite{Alaifari18a}.
		
		In finite dimensions, we present first evidence that this semi-global reconstruction regime
		allows one to do phase retrieval from measurements of bandlimited signals induced by the
		discrete Gabor transform in such a way that the corresponding stability constant only
		scales like a low order polynomial in the space dimension. To this end, we utilise
		reconstruction formulae which have become common tools in recent years
		\cite{Bojarovska15,Eldar15,Li17,Nawab83}.
	\end{abstract}

	\textbf{Keywords} Phase retrieval, Gabor frames, Time-frequency analysis
	
	\vspace{5pt}
	\textbf{Mathematics Subject Classification} 42C15, 42A38, 94A12, 65T50
	
	%------------------------------------------ Content -------------------------------------------

	\section{Introduction}\label{sec:introduction}
	Phase retrieval generally alludes to the non-linear inverse problem of recovering some signal
	(which in this paper will be modelled by $x \in \mathbb{C}^L$) from phaseless measurements.
	Some of its more well-known applications include ptychography for coherent diffraction imaging
	\cite{Humphry12,Marchesini16,Rodenburg08,Zheng13} and audio processing
	\cite{Flanagan66,Laroche99,Proakis93}. It has been shown that the phase retrieval problem for
	frames in finite-dimensional Hilbert spaces \cite{Cahill16} and a forteriori in
	finite-dimensional reflexive Banach spaces \cite{Alaifari17} is always stable, which elicits
	the question: Why are we concerned with stability estimates for phase retrieval from
	discrete Gabor measurements at all? The reason is that phase retrieval for frames in
	infinite-dimensional spaces is always unstable \cite{Cahill16, Alaifari17} and in addition one
	can construct sequences of finite-dimensional subspaces of infinite-dimensional Hilbert spaces
	along with frames for which the stability constant of phase retrieval increases exponentially
	in the dimension of the constructed subspaces \cite{Cahill16}. Recent research
	\cite{Alaifari18a} into the infinite-dimensional phase retrieval problem has however led us
	to believe that the instability of phase retrieval is not an insurmountable obstacle to
	reconstruction. It was shown that stability can be restored for examples that exhibit a
	disconnectedness in the measurements by only reconstructing the phase semi-globally or in an
	\emph{atoll sense}. Furthermore, it was shown in \cite{Grohs17} that such disconnectedness in
	the measurements is the only source of instabilities for phase retrieval.
	
	A simple example of
	instability can be obtained by considering the Gaussian functions $g(t) :=
	\mathrm{e}^{-\pi t^2}$ in conjunction with the signals
	\[
		f_\lambda^+(t) := g(t-\lambda) + g(t+\lambda) \qquad \mbox{and} \qquad f_\lambda^-(t) :=
		g(t-\lambda) - g(t+\lambda)
	\]
	depicted in figure \ref{fig:mostprominentexample}. When $\lambda$ increases, the Gaussian bumps
	in the signals $f_\lambda^\pm$ start to move further apart effectively generating what we call
	a \emph{time gap} whose length depends linearly on $\lambda$. It can be shown, see
	\cite{Alaifari18b}, that the measurements generated by the continuous Gabor transform with
	Gaussian window of the signals $f_\lambda^\pm$ have distance on the order of
	$\mathrm{e}^{-\lambda^2}$ in the standard Sobolev space $W^{1,2}(\mathbb{R}^2)$ and that one
	can therefore not stably retrieve $f_\lambda^\pm$ from continuous Gabor transform measurements.
	Similar phenomena can be observed for the discrete setting considered in this paper and we do
	therefore propose a similar paradigm as in \cite{Alaifari18a} and try to recover signals in a
	semi-global fashion that is not common in the phase retrieval literature up to this point.
	Note that in audio processing, it is natural to consider (audio) signals up to semi-global
	phase as human listeners are not able to distinguish between two signals which differ by
	semi-global phase \cite{Alaifari18a}.
	
	\begin{figure}
		\centering
		\begin{tikzpicture}
			\draw[->] (-6.3,0) -- (6.3,0) node[above] {$t$};
			\draw[->] (0,-3.3) -- (0,3.3);
			\draw[scale=2,domain=-3:3,smooth,line width=1pt, dashed] plot ({\x},{
				1.5*exp(-3.14159*(\x-2)*(\x-2))
				- 1.5*exp(-3.14159*(\x+2)*(\x+2))
			});
			\draw[scale=2,domain=-3:3,smooth, line width = 1pt] plot ({\x},{
				1.5*exp(-3.14159*(\x-2)*(\x-2))
				+ 1.5*exp(-3.14159*(\x+2)*(\x+2))
			});
			\draw[help lines, dashed] (-4,3)--(4,3);
			\draw[help lines, dashed] (-4,3)--(-4,-3);
			\draw[help lines, dashed] (4,3)--(4,0);
			\draw[help lines, dashed] (-4,-3)--(0,-3);
			\node[above right] at (-4,0) {$-\lambda$};
			\node[above left] at (4,0) {$\lambda$};
			\node at (-5.3,2.3) {$f_\lambda^+(t)$};
			\node at (-5.3,-2.3) {$f_\lambda^-(t)$};
		\end{tikzpicture}
		\caption{A simple example for instability of phase retrieval with continuous
			Gabor measurements.}
		\label{fig:mostprominentexample}
	\end{figure}
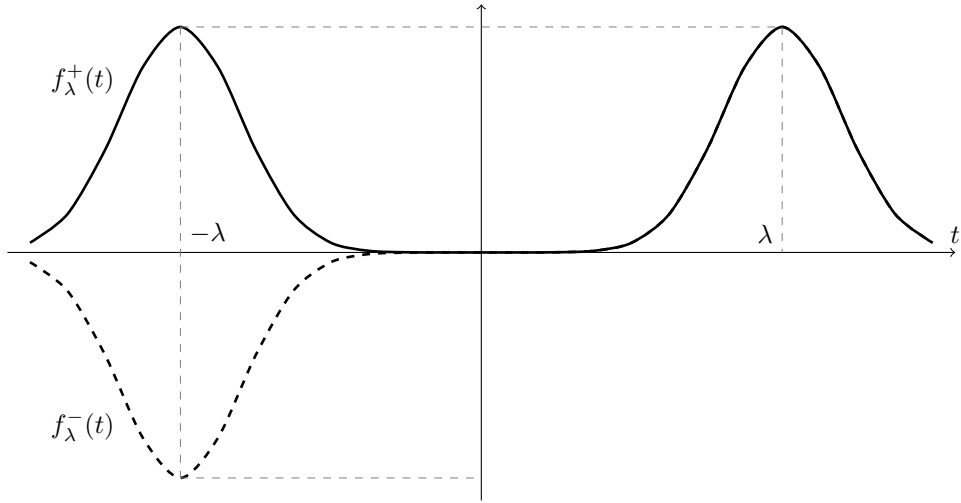
	
	One should note that in recent years a variety of stability result for phase retrieval have
	been proven. Some highlights of this research include:
	\begin{enumerate}
		\item[i.] The PhaseLift method \cite{Candes13, Candes12} which guarantees stable recovery
			from $\mathcal{O}(L)$ randomly chosen Gaussian measurements with high
			probability \cite{candes2014solving}.
		\item[ii.] The research on polarisation for phase retrieval \cite{Alexeev14, Bandeira14,
			pfander2019robust,Pfander15b} in which the authors supplement an existing measurement ensemble in order
			to obtain a phase retrieval problem that is efficiently and stably solvable.
		\item[iii.] Wirtinger flow and related methods \cite{Candes15, Sun17, Zhang17} which offer
			stability guarantees for sufficiently many randomly chosen Gaussian measurements.
		\item[iv.] The eigenvector-based angular synchronisation approach \cite{Iwen18} which
			relies on a certain weak form of invertibility of the phase retrieval problem to prove
			a stability result for deterministic measurement systems.
		\item[v.] The very recent work \cite{salanevich2019stability} in which the stability of 
			phase retrieval from (random) frames whose frame vectors are uniformly distributed on
			the unit sphere (but not necessarily independent) is considered.
	\end{enumerate}
	In some way or another, all of these results are based on different setups than ours: As
	opposed to the papers referenced in item i., iii.~and v.~we will not work with a probabilistic
	measurement system but with a deterministic one. We will also not supplement our measurement
	ensemble as is done in the results referenced in item ii.~and we will not work with the weak
	form of invertibility that is present in the paper referenced in item iv. In fact, we will
	consider the two well-known formulae \eqref{eq:autocorrelation} and \eqref{eq:ambiguity}
	presented in section \ref{sec:prerequisites} which are heavily used to develop methods for
	exact phase retrieval from Gabor measurements in the literature \cite{Eldar15,Li17,Nawab83}.
	We show that through further analysis of the formulae \eqref{eq:autocorrelation} and
	\eqref{eq:ambiguity}, one can derive stability results for some of those methods and therefore
	also for phase retrieval in general. Our stability results are designed for bandlimited signals
	and come with constants that scale in the square root of the space dimension at the cost of
	relaxing the notion of stability to resemble the one proposed in \cite{Alaifari18a}.
	
	\paragraph{Outline} In section \ref{sec:prerequisites}, we present the reader with the
	uniqueness results and the formulae on which our stability results hinge. In section
	\ref{sec:stab_ambiguity}, we utilise the ambiguity function relation \eqref{eq:ambiguity} in
	order to show that phase retrieval can be stably done for bandlimited signals based on the
	considerations in \cite{Bojarovska15, Pfander15a}. In section \ref{sec:autocorrelation}, we use
	the autocorrelation relation \eqref{eq:autocorrelation} in order to show that phase retrieval can
	be done stably for bandlimited signals utilising results from \cite{Eldar15,Li17}. As the
	proofs of our main results are a bit technical, they appear separately in section
	\ref{sec:proofs}.
	
	\section{Prerequisites}\label{sec:prerequisites}
	Throughout this paper, we fix the dimension $L \in \mathbb{N}$ and let $x,\phi \in
	\mathbb{C}^L$. We define the \emph{discrete Gabor transform (DGT)} of $x$ with window function
	$\phi$ to be
	\[
		\mathcal{V}_\phi[x](m,n) := \frac{1}{\sqrt{L}} \cdot \sum_{\ell=0}^{L-1}
			x(\ell) \overline{\phi(\ell-m)} \mathrm{e}^{-2\pi\mathrm{i} \frac{\ell n}{L}},
			\qquad m,n=0,\dots,L-1.
	\]
	Here and throughout this paper, the indexing is understood to be periodic. In particular, we
	use the convention $\phi(\ell) = \phi(\ell \mod L)$, for $\ell \in \mathbb{Z}$. A helpful way
	of looking at the DGT is to view it as a collection of windowed Fourier transforms. For this
	purpose, we denote $x_m(\ell) := x(\ell) \overline{\phi(\ell -m)}$, for $\ell,m \in
	\{0,\dots,L-1\}$, and obtain
	\[
		\mathcal{V}_\phi[x](m,n) = \mathcal{F}\left[ x_m \right](n), \qquad m,n = 0,\dots,L-1,
	\]
	where $\mathcal{F} : \mathbb{C}^L \to \mathbb{C}^L$ denotes the \emph{discrete Fourier
	transform (DFT)}
	\[
		\mathcal{F}[x](k) := \frac{1}{\sqrt{L}} \cdot \sum_{\ell=0}^{L-1} x(\ell)
			\mathrm{e}^{-2\pi\mathrm{i}\frac{\ell k}{L}}, \qquad k=0,\dots,L-1,
	\]
	with inverse
	\[
		\mathcal{F}^{-1}[x](\ell) = \frac{1}{\sqrt{L}} \cdot \sum_{k=0}^{L-1} x(k)
			\mathrm{e}^{2\pi\mathrm{i}\frac{k \ell}{L}}, \qquad \ell=0,\dots,L-1.
	\]
	We will frequently use the two-dimensional discrete Fourier transform which is the composition 
	of two DFTs as defined above.
	Additionally, we define the \emph{ambiguity function} of a signal $x$ via $\mathcal{A}[x] := 
	\mathcal{V}_x[x]$. We are interested in the recovery of signals $x \in \mathbb{C}^L$ from the
	measurements
	\[
		M_\phi[x](m,n) := \left\lvert \mathcal{V}_\phi[x](m,n) \right\rvert^2,
			\qquad m,n=0,\dots,L-1.
	\]
	It is immediately obvious that $x \in \mathbb{C}^L$ and any signal $\mathrm{e}^{\mathrm{i}
	\alpha} x$, with $\alpha \in \mathbb{R}$, yield the same measurements $M_\phi[
		\mathrm{e}^{\mathrm{i} \alpha} x
	] = M_\phi[x]$. Therefore, to have any chance of recovery, we
	will actually view $M_\phi$ as an operator defined on the quotient space $\mathbb{C}^L /
	\mathcal{S}^1$, where $\mathcal{S}^1$ denotes the unit circle. Under various assumptions, which
	we will lay out in the following, one can show that $M_\phi : \mathbb{C}^L / \mathcal{S}^1 \to
	\mathbb{R}_+^{L \times L}$ is an injective operator and that phase retrieval is therefore
	possible \emph{up to a global phase factor}. In addition, it was shown in \cite{Cahill16} that
	\[
		\inf_{\alpha \in \mathbb{R}} \lVert x - \mathrm{e}^{\mathrm{i} \alpha} y\rVert_2 \lesssim
			\lVert \left\lvert \mathcal{V}_\phi[x] \right\rvert - \left\lvert \mathcal{V}_\phi[y] \right\rvert \rVert_\mathrm{F},
	\]
	for all $x,y \in \mathbb{C}^L$, where $\lVert \cdot \rVert_\mathrm{F}$ denotes the Frobenius
	norm and the estimate depends on a constant which might increase exponentially in the space
	dimension $L$. Our phase retrieval problem is therefore possibly ill-conditioned.
	
	As mentioned before, the number of known uniqueness results has seen a stark rise in the past
	few years. In the following, we want to mention those that inspired our stability estimates.
	Let us start by remarking that almost all uniqueness results can be traced back to two
	consequential formulae which are well-known in the literature. The first of these relates the
	Gabor measurements to the autocorrelation of $x_m$. In what follows, time-reversal of a signal
	will be denoted by $x^\#(\ell) = \overline{x(-\ell)}$.
	\begin{lemma}\label{lem:autocorrelation}
		For any $x \in \mathbb{C}^L$,
		\begin{equation}\label{eq:autocorrelation}
			\mathcal{F}^{-1}\left[ M_\phi[x](m,\cdot) \right](n)
			= \frac{1}{\sqrt{L}} \cdot \left( x_m \ast x_m^\# \right)(n), \qquad m,n=0,\dots,L-1.
		\end{equation}
	\end{lemma}
	\begin{proof}\renewcommand{\qedsymbol}{}
		See appendix \ref{app:proofs_conseq_formulae}.
	\end{proof}
	The right-hand side in the above result is the aforementioned autocorrelation of $x_m$:
	\[
		\left( x_m \ast x_m^\# \right)(n) = \sum_{\ell=0}^{L-1} x(\ell)\overline{x(\ell-n)}
			\phi(\ell-n-m) \overline{\phi(\ell-m)}, \qquad m,n = 0,\dots,L-1.
	\]
	The second of these formulae relates the Gabor measurements to the ambiguity function of $x$
	and the ambiguity function of $\phi$.
	\begin{lemma}\label{lem:ambguity}
		For any $x \in \mathbb{C}^L$, the following holds:
		\begin{equation}\label{eq:ambiguity}
			\mathcal{F}\left[ M_\phi[x]\right](m,n)
			= \mathcal{A}[x](-n,m) \overline{\mathcal{A}[\phi](-n,m)}, 
			\qquad \mbox{for } m,n = 0,\dots,L-1.
		\end{equation}
	\end{lemma}
	\begin{proof}\renewcommand{\qedsymbol}{}
		See appendix \ref{app:proofs_conseq_formulae}.
	\end{proof}

	Next, we will consider the uniqueness results from \cite{Bojarovska15, Pfander15a} which are
	based on equation \eqref{eq:ambiguity}.
	
	\begin{corollary}[Theorem 2.2 in \cite{Bojarovska15}, p.~547]\label{cor:full_ambiguity_uniqueness}
		Suppose that $\phi$ satisfies
		\[
			\mathcal{A}[\phi](m,n) \neq 0, \qquad \mbox{for } m,n=0,\dots,L-1.
		\]
		Then, $x$ is uniquely determined by the measurements $M_\phi[x]$ up to global phase.
	\end{corollary}
	
	While this result is exceptionally nice in the sense that it does not impose any requirements
	on the signal, it is quite restrictive in its requirements on the window function $\phi$. For
	instance, windows $\phi$ with support length $\lvert \operatorname{supp} \phi \rvert$ smaller
	than $L/2$ will always have zero entries in their ambiguity function.
	
	\begin{corollary}[Theorem 2.4 in \cite{Bojarovska15}, p.~549]\label{cor:min_ambiguity_uniqueness}
		Let $x \in \mathbb{C}^L$ be nowhere-vanishing,
		i.e.~$\operatorname{supp} x = \{0,\dots,L-1\}$, and
		\[
			\mathcal{A}[\phi](m,n) \neq 0, \qquad \mbox{for }m=0,1,~n=0,\dots,L-1.
		\]
		Then, $x$ is uniquely determined by the measurements $M_\phi[x]$ up to global phase.
	\end{corollary}
	This result is in some sense orthogonal to corollary \ref{cor:full_ambiguity_uniqueness}: Its
	requirements on the window function are moderate while its requirements on the signal are
	rather restrictive. Of course, we might also infer a variety of results that are based on
	different trade-offs between restrictions on the window and restrictions on the signal.
	For this purpose, we introduce the parameter $\Delta \in \mathbb{N}_0$. It corresponds to the 
	maximum number of adjacent zeroes across which we may propagate phase in the reconstruction
	procedure used in the proof of the following corollary. Stated a bit more precisely: If $x$ is
	a signal of which we only know its measurements
	$M_\phi[x]$, then it follows from $\mathcal{A}[\phi](0,\cdot)$ being nowhere-vanishing (and the
	use of the ambiguity function relation) that we can reconstruct the magnitudes of $x$.
	Therefore, it suffices to propagate phase between the entries of $x$ to reconstruct $x$ up to 
	global phase. When we assume that $\mathcal{A}[\phi](m,\cdot)$ is nowhere-vanishing, for some 
	$m$, then we allow (according to the ambiguity function relation) the phase propagation from
	the entry with index $\ell$ to the entry with index $\ell + m$ (and to the entry with index
	$\ell - m$). Whether this allows us to reconstruct $x$ up to global phase depends on the set of 
	$m$ for which $\mathcal{A}[\phi](m,\cdot)$ is nowhere-vanishing and on the support set of $x$.
	This is the central idea on which the following corollary is built:
	\begin{corollary}\label{cor:bojarovska_all}
		Let $\Delta \in \mathbb{N}_0$ and let $x,y,\phi \in
		\mathbb{C}^L$ be such that $M_\phi[x] = M_\phi[y]$ and
		\[
			\mathcal{A}[\phi](m,n) \neq 0, \qquad \mbox{for } m \in \{0,\dots,\Delta+1\},~n\in\{0,\dots,L-1\}.
		\]
		Furthermore, let $G = (V,E)$ denote the graph with vertex set $V = \operatorname{supp} x$ and edge set 
		$E \subset V \times V$ such that 
		\[
			(\ell,k) \in E \Leftrightarrow \lvert \ell - k \rvert \in (0,\Delta+1] \cup [L-\Delta-1,L),
		\]
		i.e.~two vertices are connected if and only if they are at most $\Delta+1$ apart.
		If $\{V_k\}_{k=1}^K$ constitute the vertex sets of the connected components
		of $G$, then for each $k \in \{1,\dots,K\}$ there exists an $\alpha_k \in \mathbb{R}$ such
		that 
		\[
			x(\ell) = \mathrm{e}^{\mathrm{i} \alpha_k} y(\ell), \qquad \ell \in V_k.
		\]
	\end{corollary}
	\begin{proof}\renewcommand{\qedsymbol}{}
		See section \ref{sec:proofs}.
	\end{proof}

	\begin{remark}
		The corollary above is more general than corollaries \ref{cor:full_ambiguity_uniqueness}
		and \ref{cor:min_ambiguity_uniqueness}. Indeed, if $\Delta \geq \tfrac{L}{2}-1$, then $G$
		is connected. In fact, one can see from the definition of the edge set that $G$ is the
		complete graph on the vertex set $\operatorname{supp} x$. In particular, corollary
		\ref{cor:full_ambiguity_uniqueness} follows. If $\Delta = 0$ and $x$ is
		nowhere-vanishing, then $G$ is the circle graph on $L$ vertices and is thus connected. In
		this way, we recover corollary \ref{cor:min_ambiguity_uniqueness}.
	\end{remark}

	\vspace{5pt}
	Finally, we will work with a uniqueness result first proven in \cite{Eldar15} and later
	generalised in \cite{Li17} based mostly on equation \eqref{eq:autocorrelation}. Consider the
	following statement.
	\begin{corollary}[Theorem 1 in \cite{Eldar15}, p.~639]\label{cor:eldar_result}
		Let $n_0, \ell_\phi \in \{0,\dots,L-1\}$ be such that
		$\ell_\phi < L/2$ and suppose that $\ell_\phi-1$ and $L$ are coprime. If
		\[
			\operatorname{supp} \phi = \{n_0, \dots, n_0 + \ell_\phi\}
		\]
		and $\mathcal{F}[\lvert \phi \rvert^2]$ and $x$ are nowhere-vanishing, then $x$ is
		uniquely determined by the measurements $M_\phi[x]$ up to global phase.
	\end{corollary}
	The work in \cite{Li17} shows that one can also derive this result as part of a
	graph-theoretical formulation for phase retrieval.
	\begin{corollary}[Theorem 3.1 in \cite{Li17}, p.~373]\label{cor:li_result}
		Let $n_0, \ell_\phi \in \{0,\dots,L-1\}$ be such that
		$\ell_\phi < L/2$ and suppose that
		\[
			\operatorname{supp} \phi = \{n_0, \dots, n_0 + \ell_\phi\}
		\]
		and $\mathcal{F}[\lvert \phi \rvert^2]$ is nowhere-vanishing. Let the graph $G = (V,E)$
		defined by having the vertex set $V = \operatorname{supp} x$ and an edge between $\ell,
		k \in V$ if 
		\[
			\lvert \ell-k \rvert \in \{\ell_\phi, L-\ell_\phi\}
		\]
		be connected. Then, $x$ is uniquely determined by the measurements $M_\phi[x]$ up to global
		phase.
	\end{corollary}
	
	\section{Stability estimates based on the ambiguity function relation}\label{sec:stab_ambiguity}
	
	\subsection{Stability for a single island}\label{ssec:simplest_case}
	
	First, we derive stability estimates by employing equation \eqref{eq:ambiguity} and
	corollaries \ref{cor:full_ambiguity_uniqueness}--\ref{cor:bojarovska_all}. In doing this, we
	want to start with the very simple setup of corollary \ref{cor:min_ambiguity_uniqueness}.
	
	One can immediately see that there are some intricacies to the phase retrieval problem for
	signals $x \in \mathbb{C}^L$. One of those is dealing with entries $x(\ell)$ of $x$ which have
	small (or even vanishing) magnitude. For these entries, extracting the phase of $x(\ell)$ is
	unstable (or even impossible). See figure \ref{fig:smallmagnitudes} for a depiction of this
	situation. Because of this, we will mostly work with a graph capturing only the larger entries
	of the signals.
	
	\begin{figure}
		\centering
		\begin{tikzpicture}
		\draw[->] (-3.3,0) -- (3.3,0) node[above] {$\mathbb{R}$};
		\draw[->] (0,-3.3) -- (0,3.3) node[right] {$\mathrm{i} \mathbb{R}$};
		\draw[-,white] (0,0) -- (0, 0.9);
		\draw[-] (0,0) -- (-2,2) node[left] {$x(\ell)$};
		\draw[-] (-2,2) -- (2,2) node[right] {$x(\ell) + \eta$};
		\fill[black] (-2,2) circle (0.05);
		\fill[black] (2,2) circle (0.05);
		\draw[domain=0:135,smooth] plot ({0.9*cos(\x)},{0.9*sin(\x)});
		\draw[-] (0,0) -- (2,2);
		\draw[domain=0:45,smooth] plot ({.8*cos(\x)},{.8*sin(\x)});
		\node at (.5,.2) {$\beta$};
		\node at (0,.45) {$\alpha$};
		\end{tikzpicture}
		\caption{For $x(\ell), \eta \in \mathbb{C}$, the difference in absolute values satisfies
		$\lvert \lvert x(\ell) \rvert - \lvert x(\ell) + \eta\rvert \rvert \leq \lvert \eta \rvert$
		such that the map $\lvert \cdot \rvert : \mathbb{C} \to \mathbb{R}_+$ can be seen to be
		stable. On the other hand, the function which maps complex numbers to their phase is
		unstable at the origin: To see this, we can choose $x(\ell) = (-1+\mathrm{i})\epsilon$,
		$\eta = 2\epsilon$ such that $\lvert \alpha - \beta \rvert = \pi / 2 \geq \pi / (4\epsilon)
		\cdot \lvert \eta \rvert$, where $\alpha, \beta \in (-\pi,\pi]$ denote the principal values
		of the arguments of $x(\ell), x(\ell) + \eta \in \mathbb{C}$, respectively.}
		\label{fig:smallmagnitudes}
	\end{figure}
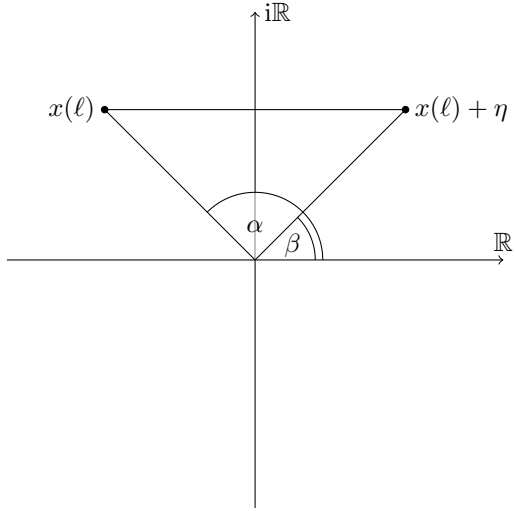
	
	\begin{definition}
		\label{def:ess_supp}
		Let $\Delta \in \mathbb{N}_0$ and $\delta > 0$. We call the graph $G_\delta =
		(V_\delta,E)$ defined by having the vertex set
		\[
			V_\delta = \{ \ell \in \{0,\dots,L-1\} \;\vert\; \lvert x(\ell) \rvert \geq \delta \}
		\]
		and an edge between $\ell,k \in V$ if 
		\[
			\lvert \ell - k \rvert \in (0,\Delta+1] \cup [L-\Delta-1,L),
		\]
		the \emph{essential support graph of $x$ with time separation parameter $\Delta$}. We will 
		also simply call the essential support graph of $x$ with time separation parameter zero
		the \emph{essential support graph of $x$}.
	\end{definition}

	The stability estimates we derive hold for bandlimited signals defined as follows:

	\begin{definition}
		Let $B \in \mathbb{N}_0$. We call $x \in \mathbb{C}^L$ \emph{$B$-bandlimited} if
		\[
			\operatorname{supp} \mathcal{F}[x] \subset \{-B,\dots,B\} \mod L.
		\]
	\end{definition}

	One important property of bandlimited signals is that their ambguity function does not have
	full support (when $4B < L$).

	\begin{proposition}
		\label{prop:bandlimited_amb}
		Let $x \in \mathbb{C}^L$ be $B$-bandlimited, for some $B \in \mathbb{N}_0$. Then, it
		holds for all $m \in \{0,\dots,L-1\}$ that
		\[
			\operatorname{supp} \mathcal{A}[x](m,\cdot) \subset \{-2B,\dots,2B\} \mod L.
		\]
	\end{proposition}

	We included a proof of this basic proposition in appendix \ref{app:proofs_conseq_formulae}. In
	the following, we will work with the $\ell^2$-norm on subsets $S$ of $\{0,\dots,L-1\}$. For 
	such sets, we define
	\[
		\lVert x \rVert_{\ell^2(S)} := \left( \sum_{\ell \in S} \lvert x(\ell) \rvert^2 \right)^\frac12.
	\]
	We may now prove the following result on the stability of magnitude retrieval:
	\begin{lemma}[Stability of magnitude retrieval]\label{lem:magnitude_stab}
		Let $\delta > 0$ and let $x,y \in \mathbb{C}^L$ be $B$-bandlimited, for some $B \in
		\mathbb{N}_0$. Define $G_\delta = (V_\delta,E)$ to be the essential 
		support graph of $x$ and let $\phi \in \mathbb{C}^L$ be such that
		\[
			\min_{n \in \{-2B,\dots,2B\}} \lvert \mathcal{A}[\phi](0,n) \rvert \geq \frac{1}{c},
		\]
		for some $c > 0$. Then, 
		\[
			\left\lVert \lvert x \rvert - \lvert y \rvert \right\rVert_{\ell^2(V_\delta)} \leq
			\frac{c}{\delta} \left\lVert M_\phi[x]-M_\phi[y] \right\rVert_\mathrm{F}.
		\]
	\end{lemma}
	\begin{proof}
		Let us start with the simple fact that for $\ell \in V_\delta$ (cf.~definition
		\ref{def:ess_supp}),
		\[
			\left\lvert \lvert x(\ell) \rvert - \lvert y(\ell) \rvert \right\rvert
			= \frac{\left\lvert \lvert x(\ell) \rvert^2 - \lvert y(\ell) \rvert^2 \right\rvert}
				{\left\lvert \lvert x(\ell) \rvert + \lvert y(\ell) \rvert \right\rvert}
			\leq \frac{1}{\delta} \left\lvert \lvert x(\ell) \rvert^2 - \lvert y(\ell) \rvert^2 \right\rvert.
		\]
		Thus, we have
		\begin{equation*}
			\left\lVert \lvert x \rvert - \lvert y \rvert \right\rVert_{\ell^2(V_\delta)}
			\leq \frac{1}{\delta} \left\lVert \lvert x \rvert^2 - \lvert y \rvert^2 \right\rVert_{\ell^2(V_\delta)}
			\leq \frac{1}{\delta} \left\lVert \lvert x \rvert^2 - \lvert y \rvert^2 \right\rVert_2.
		\end{equation*}
		Next, suppose that $4B < L$. Employing Plancherel's theorem, the ambiguity function relation
		and proposition \ref{prop:bandlimited_amb}, we find 
		\begin{align*}
			\left\lVert \lvert x \rvert^2 - \lvert y \rvert^2 \right\rVert_2^2
			&= \left\lVert \mathcal{A}[x](0,\cdot) - \mathcal{A}[y](0,\cdot) \right\rVert_2^2
			= \sum_{n = 0}^{L-1} \left\lvert \mathcal{A}[x](0,n) - \mathcal{A}[y](0,n) \right\rvert^2 \\
			&= \sum_{n = -2B}^{2B} \left\lvert \mathcal{A}[x](0,n) - \mathcal{A}[y](0,n) \right\rvert^2
			= \sum_{n = -2B}^{2B} \left\lvert
				\frac{\mathcal{F}\left[ M_\phi[x] - M_\phi[y] \right](n,0)}{\mathcal{A}[\phi](0,n)}
			\right\rvert^2 \\
			&\leq c^2 \sum_{n = -2B}^{2B} \left\lvert \mathcal{F}\left[ M_\phi[x] - M_\phi[y] \right](n,0) \right\rvert^2
			\leq c^2 \left\lVert \mathcal{F}\left[ M_\phi[x] - M_\phi[y] \right] \right\rVert_2^2 \\
			&= c^2 \left\lVert M_\phi[x] - M_\phi[y] \right\rVert_2^2.
		\end{align*}
		The proof for $4B \geq L$ is even simpler: In this case $\mathcal{A}[\phi](0,\cdot)$ is 
		lower bounded everywhere.
	\end{proof}
	
	\begin{remark}
		Note that we need to restrict our stability result to the essential support $V_\delta$ of
		$x$ because the square root $t \mapsto \sqrt{t}$ is not Lipschitz continuous. For this reason, 
		we obtain the dependence of our stability result on $\delta$. We note that in the above
		proof, we derive 
		\[
			\left\lVert \lvert x \rvert^2 - \lvert y \rvert^2 \right\rVert_2
			\leq c \left\lVert M_\phi[x] - M_\phi[y] \right\rVert_2.
		\]
		Hence, magnitude retrieval is, in fact, stable even when we consider small entries as long as we compare
		the squared magnitudes of the signals with the squared absolute values of the Gabor transform.
	\end{remark}

	Next, we turn to the retrieval of the phases. First, in accordance with corollary
	\ref{cor:min_ambiguity_uniqueness}, we will only use the entries of $\mathcal{A}[x](1,\cdot)$
	(and of $\mathcal{A}[x](0,\cdot)$) for our recovery which allows us to do phase propagation on
	adjacent entries. To be precise, we can propagate the phase from $x(\ell)$ to $x(\ell+1)$ (or
	back), for any $\ell \in \{0,\dots,L-1\}$. Mathematically this fact can be captured with the
	help of the essential support graph $G_\delta$ of $x$ with time-separation parameter zero,
	i.e.~$\Delta = 0$. In the following, we will call the connected components of $G_\delta$
	\emph{temporal islands}.
	
	\begin{theorem}[Stability of phase retrieval on a single temporal island]\label{thm:ambig_one_island}
		Let $\delta > 0$ and let $x,y \in \mathbb{C}^L$ be $B$-bandlimited, for $B \in \mathbb{N}_0$.
		For $G_\delta = (V_\delta, E)$ the essential support graph of $x$, assume that the subgraph
		induced by the vertex set $S_\delta = V_\delta \cap \operatorname{supp} y$ is connected. If
		$\phi \in \mathbb{C}^L$ is such that 
		\[
			\min_{\substack{ m \in \{0,1\}\\n \in \{-2B,\dots,2B\}}} \lvert \mathcal{A}[\phi](m,n) \rvert
			\geq \frac{1}{c},
		\]
		for some $c > 0$, then 
		\[
			\inf_{\alpha \in \mathbb{R}} \left\lVert
				x - \mathrm{e}^{\mathrm{i} \alpha} y
			\right\rVert_{\ell^2(V_\delta)} 
			\leq \frac{c}{\delta} \left( 1 +
				\frac{\sqrt{2 \lvert S_\delta \rvert} \lVert x \rVert_{\ell^2(S_\delta)}}{\delta} \right) \left\lVert 
				M_\phi[x] - M_\phi[y]
			\right\rVert_{\mathrm{F}}.
		\]
	\end{theorem}
	
	\begin{proof}\renewcommand{\qedsymbol}{}
		See section \ref{sec:proofs}.
	\end{proof}
	
	\begin{remark}
		The stability constant derived in the above result is
		\[
			\frac{c}{\delta} \left( 1 +
				\frac{\sqrt{2 \lvert S_\delta \rvert} \lVert x \rVert_{\ell^2(S_\delta)}}{\delta} \right)
		\]
		and consists of a contribution from the magnitude retrieval estimate in lemma
		\ref{lem:magnitude_stab} and the phase retrieval estimate presented in section
		\ref{sec:proofs}.
		The contribution from the phase retrieval estimate contains a factor
		$\sqrt{2 \lvert S_\delta \rvert}$ which can be interpreted as a mild ill-conditioning of
		phase retrieval as it might scale like $L^\frac12$. Additionally, the phase retrieval
		estimate contains a factor $1/\delta$ which captures the dependence on the size of the
		magnitudes of $x$.
	\end{remark}
	
	\begin{figure}
		\begin{subfigure}[b]{.475\linewidth}
			\centering
			\includegraphics[width=\textwidth]{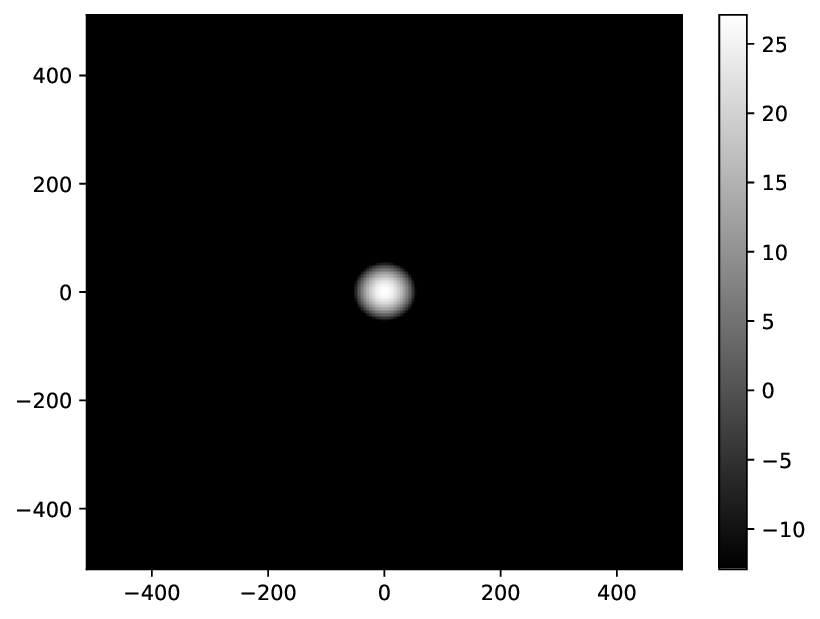}
			\caption{The ambiguity function of the Gaussian window.}
		\end{subfigure}~
		\begin{subfigure}[b]{.475\linewidth}
			\centering
			\includegraphics[width=\textwidth]{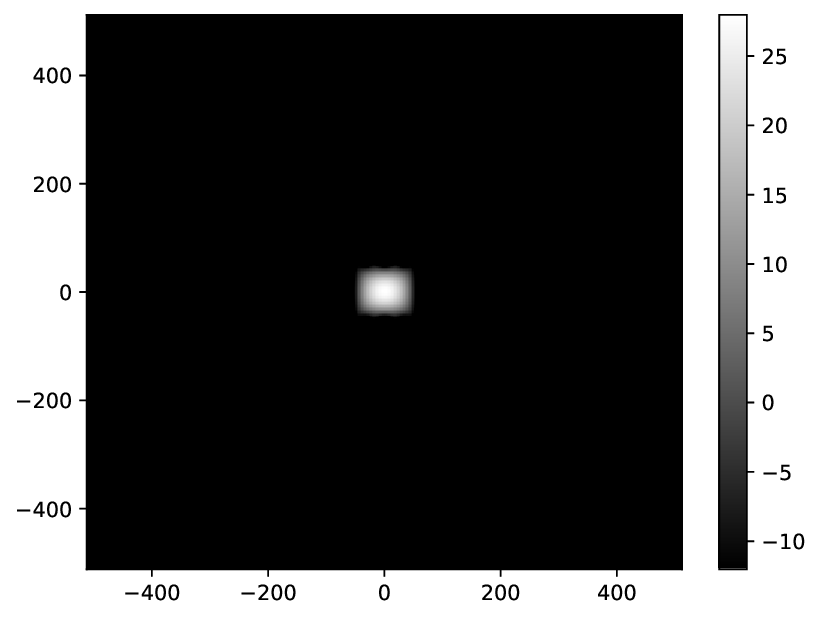}
			\caption{The ambiguity function of the Hamming window.}
		\end{subfigure}
		\\
		\begin{subfigure}[b]{.475\linewidth}
			\centering
			\includegraphics[width=\textwidth]{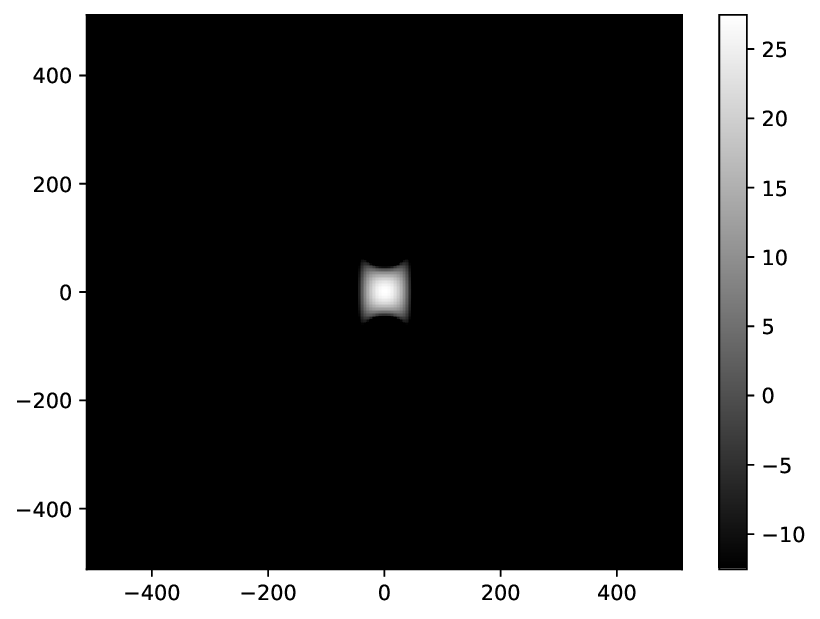}
			\caption{The ambiguity function of the Hann window.}
		\end{subfigure}~
		\begin{subfigure}[b]{.475\linewidth}
			\centering
			\includegraphics[width=\textwidth]{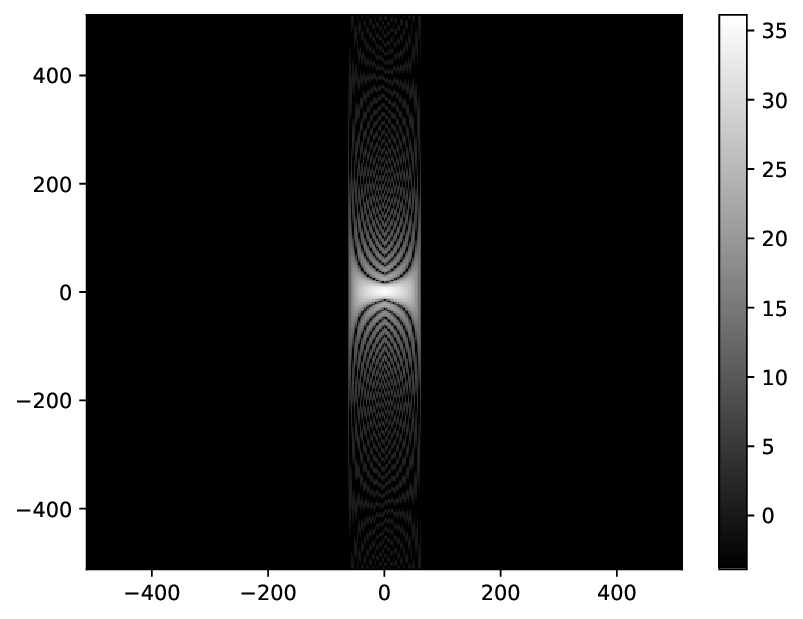}
			\caption{The ambiguity function of the rectangular window.}
		\end{subfigure}
		\caption{Visualisation of the magnitudes of the ambiguity functions of some commonly used
		window functions in a logarithmic scale (we plot $20 \log_{10} \lvert \mathcal{A}[\phi] \rvert$).} \label{fig:ambiguity_function}
	\end{figure}
	
	For a visualisation, we plot the magnitudes ambiguity functions of four commonly used window
	functions $\phi \in \mathbb{C}^L$ in figure \ref{fig:ambiguity_function}. For reference, we use
	$L = 1024$ and the windows
	
	\begin{gather*}
		\phi_{\mathrm{gauss}}(\ell) = \mathrm{e}^{-\pi \frac{(\ell-512)^2}{32^2}},
		\qquad \phi_{\mathrm{hamming}}(\ell) := \begin{cases}
			\frac{25}{46} - \frac{21}{46} \cos\left( \frac{2 \pi \ell}{63} \right) & \mbox{if } \ell = 0,\dots,63, \\
			0 & \mbox{else},
		\end{cases}
		\\
		\phi_\mathrm{hann}(\ell) := \begin{cases}
			\frac12 - \frac12 \cos\left( \frac{2 \pi \ell}{63} \right) & \mbox{if } \ell = 0,\dots,63, \\
		0 & \mbox{else},
		\end{cases} \qquad \phi_{\mathrm{rectangular}}(\ell) := \begin{cases}
			1 & \mbox{if } \ell = 0,\dots,63, \\
			0 & \mbox{else}.
		\end{cases}
	\end{gather*}

	\begin{example}
		We want to present an example to clarify the statement of Theorem
		\ref{thm:ambig_one_island}. For this purpose, we let $L \in \mathbb{N}$, with $L \geq 6$, be
		arbitrary but fixed and consider the rectangular window of length two 
		\[
			\phi(\ell) := \begin{cases}
				1 & \mbox{if } \ell = 0,1, \\
				0 & \mbox{else}.
			\end{cases}
		\]
		Note that the choice of the rectangular window of length two is rather arbitrary. One
		could, in fact, perform similar calculations with most other common window functions as 
		long as one picks the window in such a way that its time support is small enough. We 
		observe that, for large $L$, the rectangular window of length two will have a small time
		support (by which we mean that $\lvert \operatorname{supp} \phi \rvert = 2$ is 
		small compared to $L$) and a large frequency support (by which we mean that $\lvert
		\operatorname{supp} \mathcal{F}[\phi] \rvert$, which is readily seen to be $L$ or $L-1$ 
		depending on whether $L$ is even or odd, is comparable to $L$). This property of the 
		rectangular window of length two will carry over to its ambiguity function in the sense 
		that even for large $B$, we find that 
		\[
			\min_{\substack{ m \in \{0,1\}\\n \in \{-2B,\dots,2B\}}} \lvert \mathcal{A}[\phi](m,n) \rvert
		\]
		is rather large and thereby the constant $c$ in Theorem \ref{thm:ambig_one_island} is 
		rather small. So let $\widetilde c > 1$ and $B \in \mathbb{N}_0$ be such that
		$\tfrac{L}{\widetilde c} \leq 6B \leq L$. Then, we find that 
		\[
			\min_{\substack{ m \in \{0,1\}\\n \in \{-2B,\dots,2B\}}} \lvert \mathcal{A}[\phi](m,n) \rvert
			= \frac{1}{\sqrt{L}}.
		\]
		Therefore, it follows that Theorem \ref{thm:ambig_one_island} holds with $c = \sqrt{L}$,
		i.e.~for $\delta > 0$ and $x,y \in \mathbb{C}^L$ $B$-bandlimited such that the subgraph of 
		the essential support graph of $x$ induced by the vertex set $S_\delta = V_\delta \cap
		\operatorname{supp} y$ is connected, we have 
		\[
			\inf_{\alpha \in \mathbb{R}} \left\lVert
				x - \mathrm{e}^{\mathrm{i} \alpha} y
			\right\rVert_{\ell^2(V_\delta)} 
			\leq \frac{\sqrt{L}}{\delta} \left( 1 +
				\frac{\sqrt{2 \lvert S_\delta \rvert} \lVert x \rVert_{\ell^2(S_\delta)}}{\delta} \right) \left\lVert 
				M_\phi[x] - M_\phi[y]
			\right\rVert_{\mathrm{F}}.
		\]
		It follows that for the rectangular window of size two and for $6B \leq L$, the stability 
		constant scales linearly in $L$. The dimension of the space of $B$-bandlimited signals is 
		$d:=2B+1$ and therefore it follows from $\tfrac{L}{\widetilde c} \leq 6B$ that
		\[
			\inf_{\alpha \in \mathbb{R}} \left\lVert
				x - \mathrm{e}^{\mathrm{i} \alpha} y
			\right\rVert_{\ell^2(V_\delta)} 
			\leq \frac{\sqrt{6(d-1)}}{\delta} \left( 1 +
				\frac{\sqrt{6\widetilde{c} (d-1)} \lVert x \rVert_{\ell^2(S_\delta)}}{\delta} \right) \left\lVert 
				M_\phi[x] - M_\phi[y]
			\right\rVert_{\mathrm{F}}.
		\]
	\end{example}

	\begin{example}
		We also want to give an example which elaborates on the dependence of the stability constant on $\delta^{-1}$.
		On first glance, it might look like this dependence is only due to our analysis (the
		propagation of phases between adjacent entries.) However, this is not the case. Consider
		the rectangular window $\phi$ of length two along with the signals 
		\[
			x(\ell) = \begin{cases}
				1 & \mbox{if } \ell = 0, \\
				\delta & \mbox{if } \ell = 1, \\
				1 & \mbox{if } \ell = 2, \\
				0 & \mbox{else,}
			\end{cases} \qquad y(\ell) = \begin{cases}
				1 & \mbox{if } \ell = 0, \\
				\delta & \mbox{if } \ell = 1, \\
				-1 & \mbox{if } \ell = 2, \\
				0 & \mbox{else,}
			\end{cases} 
		\]
		where $\delta \in (0,1)$. Then, we have 
		\[
			\inf_{\alpha \in \mathbb{R}} \left\lVert x - \mathrm{e}^{\mathrm{i} \alpha} y \right\rVert_2 = 2.
		\]
		In addition, we have 
		\[
			\left\lVert M_\phi[x] - M_\phi[y] \right\rVert_\mathrm{F} = 2 \sqrt{2} \frac{\delta}{\sqrt{L}}.
		\]
		Therefore, it follows that
		\begin{equation}
			\label{eq:delta}
			\frac{\inf_{\alpha \in \mathbb{R}} \left\lVert x - \mathrm{e}^{\mathrm{i} \alpha} y \right\rVert_2}{\left\lVert M_\phi[x] - M_\phi[y] \right\rVert_\mathrm{F}}
				= \sqrt{\frac{L}{2}} \cdot \delta^{-1}
		\end{equation}
		is a lower bound for the stability constant. Note that this example is independent of the 
		reconstruction technique we have chosen.
	\end{example}
		
	\subsection{Multiple islands and frequency gaps}\label{ssec:multiple_islands_ambig}
	The phase propagation procedure presented as part of the proof of theorem
	\ref{thm:ambig_one_island} carries over quite naturally to the case where the graph $G_\delta =
	(V_\delta,E)$ is disconnected rather than connected. We say the graph $G_\delta$ has
	\emph{multiple temporal islands}. It is of course interesting to consider this case, as there
	is a wide range of signals for which $G$ will be disconnected. For instance, recordings of
	human speech will typically consist of multiple temporal islands as speakers tend to leave
	short gaps (i.e.~modes of silence) in between words. In addition, a discretisation of the
	signal $f_\lambda^+$ from the introduction (see figure \ref{fig:discretised}) will yield two
	temporal islands.
	
	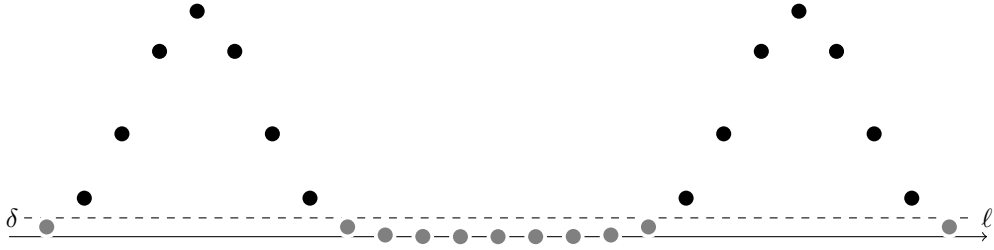
\begin{figure}
		\centering
		\begin{tikzpicture}
		\draw[->] (-6.5,0) -- (6.5,0) node[above] {$\ell$};
		\draw[dashed] (-6.3,0.25) -- (6.3,0.25);
		\foreach \i in {-2,...,22}{
			
			\pgfmathsetmacro\height{
					3*exp(-3.14159*((\i-10)/4-2)*((\i-10)/4-2)) + 3*exp(-3.14159*((\i-10)/4+2)*((\i-10)/4+2))
				}
			
			\fill[white] ({(\i-10)/2},{\height}) circle (0.15);
			
			\pgfmathparse{\height < 0.25 ? int(1) : int(0)}
			\ifnum\pgfmathresult=1
				\fill[gray] ({(\i-10)/2},{\height}) circle (0.1);
			\else
				\fill[black] ({(\i-10)/2},{\height}) circle (0.1);
			\fi
			}
		
		\node[left] at (-6.25,0.25) {$\delta$};
		\end{tikzpicture}
		\caption{The function $f_\lambda^+$ from the introduction after discretisation. Entries of
		the resulting signal that fall below a certain threshold $\delta > 0$ are coloured in grey.
		The remaining entries are coloured in black and make up the vertex set $V_\delta$. In this
		picture, we can clearly see the two temporal islands.}
		\label{fig:discretised}
	\end{figure}

	\begin{theorem}[Stability of phase retrieval on multiple temporal islands]\label{thm:unnec}
		Let $\delta > 0$ and let $x,y \in \mathbb{C}^L$ be $B$-bandlimited, for $B \in
		\mathbb{N}_0$. For the essential support graph $G_\delta = (V_\delta, E)$ of $x$, assume 
		that the subgraph induced by the vertex set $S_\delta = V_\delta \cap
		\operatorname{supp} y$ has $K$ connected components whose vertex sets are denoted by
		$\{S_k\}_{k=1}^K$. If $\phi \in \mathbb{C}^L$ is such that 
		\[
			\min_{\substack{ m \in \{0,1\}\\n \in \{-2B,\dots,2B\}}} \lvert \mathcal{A}[\phi](m,n) \rvert
			\geq \frac{1}{c},
		\]
		for some $c > 0$, then
		\[
			\inf_{\alpha_1,\dots,\alpha_K \in \mathbb{R}} \sum_{k=1}^K \left\lVert 
				x - \mathrm{e}^{\mathrm{i} \alpha_k} y 
			\right\rVert_{\ell^2(S_k)} \leq \frac{c \sqrt{K}}{\delta} \left( 1 + 
			\frac{\sqrt{ 2 \lvert S_\delta \rvert } \lVert x \rVert_{\ell^2(S_\delta)}}{\delta} \right)
			\left\lVert M_\phi[x] - M_\phi[y] \right\rVert_{\mathrm{F}}.
		\]
	\end{theorem}

	\begin{proof}\renewcommand{\qedsymbol}{}
		See theorem \ref{thm:fullstab_I}.
	\end{proof}

	Furthermore, one should note that until now we have only worked with minimal restrictions on
	the ambiguity function $\mathcal{A}[\phi]$ of the window $\phi$, i.e.~we have only utilised the
	ambiguity function for $m=0,1$. In the following, we want to generalise our result to be able
	to use $\mathcal{A}[\phi](m,n)$ for $m=0,\dots,\Delta+1$. In particular, we may be able to harness
	corollary \ref{cor:bojarovska_all} in order to propagate phase stably across a section of the
	signal in which the entries consistently fall below a threshold $\delta$. To precisely describe
	this phase propagation procedure, we make use of the essential support graph of signals $x$
	with time-separation parameter $\Delta$.
	
	\begin{theorem}[Main theorem]\label{thm:fullstab_I}
		Let $\Delta \in \mathbb{N}_0$, let $\delta > 0$ and suppose that $x,y \in
		\mathbb{C}^L$ are $B$-bandlimited, for $B \in \mathbb{N}_0$. Let $G_\delta =
		(V_\delta, E)$ be the essential support graph of $x$ with time-separation parameter
		$\Delta$ and assume that the subgraph induced by the vertex set $S_\delta = V_\delta \cap
		\operatorname{supp} y$ has $K$ connected components whose vertex sets are denoted by
		$\{S_k\}_{k=1}^K$. If $\phi \in \mathbb{C}^L$ is such that 
		\begin{equation}
			\label{eq:low_bound_amb}
			\min_{\substack{ m \in \{0,\dots,\Delta+1\}\\n \in \{-2B,\dots,2B\}}} \lvert \mathcal{A}[\phi](m,n) \rvert
			\geq \frac{1}{c},
		\end{equation}
		for some $c > 0$, then
		\[
			\inf_{\alpha_1,\dots,\alpha_K \in \mathbb{R}} \sum_{k=1}^K \left\lVert 
				x - \mathrm{e}^{\mathrm{i} \alpha_k} y 
			\right\rVert_{\ell^2(S_k)} \leq \frac{c \sqrt{K}}{\delta} \left( 1 + 
			2 \sqrt{\frac{L+\Delta}{2+\Delta} } \frac{\lVert x \rVert_{\ell^2(S_\delta)}}{\delta} \right)
			\left\lVert M_\phi[x] - M_\phi[y] \right\rVert_{\mathrm{F}}.
		\]
	\end{theorem}
	\begin{proof}\renewcommand{\qedsymbol}{}
		See section \ref{sec:proofs}.
	\end{proof}

	\begin{remark}
		Alternatively, one can show
		\[
			\inf_{\alpha_1,\dots,\alpha_K \in \mathbb{R}} \sum_{k=1}^K \left\lVert 
				x - \mathrm{e}^{\mathrm{i} \alpha_k} y 
			\right\rVert_{\ell^2(S_k)} \leq \frac{c}{\delta} \left( \sqrt{K} + 
			\sum_{k=1}^K \frac{\sqrt{2 \lvert S_k \rvert} \lVert x \rVert_{\ell^2(S_k)}}{\delta} \right)
			\left\lVert M_\phi[x] - M_\phi[y] \right\rVert_2
		\]
		under the assumptions laid out in theorem \ref{thm:fullstab_I}. We prefer the result above 
		as it is more compact. Note that neither of these results is stronger or weaker than the
		other.
	\end{remark}
	
	We note that one can obtain dual results to the above by considering $\mathcal{F}[x]$ instead of $x$.
	A straight-forward calculation (see proof of proposition \ref{prop:bandlimited_amb}) yields
	\[
		\mathcal{A}[x](m,n) = \mathrm{e}^{-2\pi\mathrm{i}\frac{mn}{L}} \mathcal{A}[\mathcal{F}[x]](n,-m),
		\qquad m,n = 0,\dots,L-1.
	\]
	In this light, it is not surprising that we can derive stability results for
	recovering $\mathcal{F}[x]$ from the measurements $M_\phi[x]$ resembling the theorems derived
	above. Note that in this way, one can also show the following
	dual of the ambguity function relation:
	\begin{equation*}
		\mathcal{F}\left[ M_\phi[x]\right](m,n)
		= \mathcal{A}[\mathcal{F}[x]](m,n) \overline{\mathcal{A}[\mathcal{F}[\phi]](m,n)},
		\qquad m,n = 0,\dots,L-1.
	\end{equation*}

	\begin{theorem}[Stability for frequency gaps]\label{thm:stab_freq_I}
		Let $B, \Delta \in \mathbb{N}_0$, let $\delta > 0$ and suppose that $x,y \in
		\mathbb{C}^L$ have their support contained in $\{-B,\dots,B\} \mod L$. Let $G_\delta =
		(V_\delta, E)$ be the essential support graph of $\mathcal{F}[x]$ with time-separation
		parameter $\Delta$ and assume that the subgraph induced by the vertex set $S_\delta =
		V_\delta \cap \operatorname{supp} \mathcal{F}[y]$ has $K$ connected components whose vertex
		sets are denoted by $\{S_k\}_{k=1}^K$. If $\phi \in \mathbb{C}^L$ is such that 
		\[
			\min_{\substack{ m \in \{-2B,\dots,2B\}\\n \in \{0,\dots,\Delta+1\}}} \lvert \mathcal{A}[\phi](m,n) \rvert
			\geq \frac{1}{c},
		\]
		for some $c > 0$, then
		\[
			\inf_{\alpha_1,\dots,\alpha_K \in \mathbb{R}} \sum_{k=1}^K \left\lVert 
				\mathcal{F}[x] - \mathrm{e}^{\mathrm{i} \alpha_k} \mathcal{F}[y]
			\right\rVert_{\ell^2(S_k)} \leq \frac{c \sqrt{K}}{\delta} \left( 1 + 
			2 \sqrt{\frac{L+\Delta}{2+\Delta}} \frac{\lVert x \rVert_2}{\delta} \right)
			\left\lVert M_\phi[x] - M_\phi[y] \right\rVert_2.
		\]
	\end{theorem}
	
	\begin{remark}
		In the preceding pages, we have presented approaches for phase retrieval for signals with 
		multiple temporal or frequency islands. Unfortunately, it is not so clear how to extend
		this work to the more general case of time-frequency atolls considered in
		\cite{Alaifari18a,Grohs17}. It is likely that one has to come up with a different approach that
		allows one to do phase propagation in frequency and time direction simultaneously to
		actually handle time-frequency atolls.
	\end{remark}
	
	We want to end this section by remarking that from our proof strategy for the frequency result
	a straight-forward dual version of corollary \ref{cor:bojarovska_all} follows.
	
	\begin{corollary}
		Let $\Delta \in \mathbb{N}_0$ and let $x,y,\phi \in \mathbb{C}^L$ be such that $M_\phi[x]
		= M_\phi[y]$. Assume that the window $\phi$ satisfies
		\[
			\mathcal{A}[\phi](m,n) \neq 0, \qquad \mbox{for } m \in \{0,\dots,L-1\},~n\in\{0,\dots,\Delta+1\}.
		\]
		Define $G = (V,E)$ as the graph with vertex set $V = \operatorname{supp} \mathcal{F}[x]$
		and edge set $E \subset \{0,\dots,L-1\} \times \{0,\dots,L-1\}$ given by
		\[
			(\ell,k) \in E \Leftrightarrow \lvert \ell - k \rvert \in (0,\Delta+1] \cup [L-\Delta-1,L).
		\]
		If $\{V_k\}_{k=1}^K$ are the vertex sets of the connected components of $G$, then for all
		$k \in \{1,\dots,K\}$ there exists an $\alpha_k \in \mathbb{R}$ such that 
		\[
			\mathcal{F}[x](\ell)
			= \mathrm{e}^{\mathrm{i} \alpha_k} \mathcal{F}[y](\ell), \qquad \ell \in V_k.
		\]
	\end{corollary}
	
	\section{Stability estimates based on the autocorrelation relation}\label{sec:autocorrelation}
	The goal of this section is to apply the techniques we have developed thus far to the setup
	proposed in \cite{Li17}. Our approach will be designed to work for bandlimited signals $x \in
	\mathbb{C}^L$ which potentially have very small entries. In doing so, we do not need to require
	that the Fourier transform of the absolute value squared of the window is nowhere-vanishing.
	We emphasise these particularities as the stability result developed in \cite{Li17} (theorem
	4.2 on p.~375) relies on window functions for which $\mathcal{F}[\lvert \phi \rvert^2]$ is
	nowhere-vanishing, and while one may apply it to signals $x \in \mathbb{C}^L$ with very small
	entries, the resulting stability constant will be ill-behaved as it depends inversely on
	$\min_{\ell \in \operatorname{supp} x} \lvert x(\ell) \rvert^2$.
	
	Recall from section \ref{sec:prerequisites} that the authors of \cite{Li17} consider the
	graph $G = (V,E)$ with vertex set $V := \operatorname{supp}(x)$ and an edge between $\ell,k \in
	V$ if
	\[
		\lvert \ell-k \rvert \in \{\ell_\phi, L-\ell_\phi\},
	\]
	where $\ell_\phi+1$ denotes the support length of the window function. They do then propose to
	reconstruct the magnitude of a signal $x \in \mathbb{C}^L$ using the ambiguity function
	relation \eqref{eq:ambiguity} (as in section \ref{sec:stab_ambiguity}) and propagate phase from
	one entry of $x$ to another if their indices have distance $\ell_\phi$ using equation
	\eqref{eq:autocorrelation}. As before, we will work with local lower bounds on the ambiguity
	function of the window and the signal in order to ensure that all the aforedescribed steps
	can be carried out stably.
	
	In order to introduce the local lower bounds on the signal, we will have to modify the graphs
	presented in \cite{Li17} slightly: Let us consider a signal $x \in \mathbb{C}^L$, a tolerance
	parameter $\delta > 0$ and a window function $\phi \in \mathbb{C}^L$ such that
	\[
		\operatorname{supp}(\phi) = \{n_0, \dots, n_0+\ell_\phi\} \mod L,
	\]
	for $n_0, \ell_\phi \in \{0,\dots,L-1\}$ such that $2 \ell_\phi < L$. Introduce the graph
	$G_\delta = (V_\delta,E)$ with vertex set
	\begin{equation*}
		V_\delta :=
		\left\lbrace \ell \in \{0,\dots,L-1\} \;\middle\vert\; \lvert x(\ell) \rvert \geq \delta \right\rbrace
	\end{equation*}
	and an edge between $\ell,k \in V$ if
	\begin{equation*}
		\lvert \ell - k \rvert \in \{\ell_\phi, L-\ell_\phi\}.
	\end{equation*}
	Under these assumptions, we can state the following stability estimate whose proof is inspired
	by the proof of corollary \ref{cor:li_result}:
	
	\begin{theorem}\label{thm:stabli_connected}
		Let $\ell_\phi,n_0 \in \mathbb{N}_0$ such that $2 \ell_\phi < L$. Furthermore, let
		$\delta > 0$ and let $x,y \in \mathbb{C}^L$ be $B$-bandlimited, for $B \in \mathbb{N}_0$.
		Suppose that the subgraph induced by the vertex set $S_\delta = V_\delta \cap
		\operatorname{supp} y$ is connected and that the window $\phi$ satisfies
		\[
			\operatorname{supp}(\phi) = \{n_0, \dots, n_0+\ell_\phi\} \mod L
		\]
		as well as 
		\[
			\min_{k \in \{-2B,\dots,2B\}} \left\lvert
				\mathcal{F}\left[ \lvert \phi \rvert^2 \right](k)
			\right\rvert \geq \frac{1}{c},
		\]
		for some $c > 0$. Then,
		\begin{equation*}
			\inf_{\alpha \in \mathbb{R}} \left\lVert
				x - \mathrm{e}^{\mathrm{i} \alpha} y
			\right\rVert_{\ell^2(V_\delta)}
			\leq \frac{1}{\delta} \left( c
				+ \frac{2 \sqrt{\lvert S_\delta \rvert L} \lVert x \rVert_{\ell^2(S_\delta)}}{\delta
				\lvert \phi(n_0) \phi(n_0 + \ell_\phi) \rvert} \right)
				\left\lVert M_\phi[x]-M_\phi[y] \right\rVert_\mathrm{F}.
		\end{equation*}
	\end{theorem}
	\begin{proof}\renewcommand{\qedsymbol}{}
		See section \ref{sec:proofs}.
	\end{proof}
	\begin{remark}
		The stability constant in this result is
		\[
			\frac{1}{\delta} \left( c
				+ \frac{2 \sqrt{\lvert S_\delta \rvert L} \lVert x \rVert_{\ell^2(S_\delta)}}{\delta
				\lvert \phi(n_0) \phi(n_0 + \ell_\phi) \rvert} \right).
		\]
		The part of the constant due to magnitude retrieval is exactly the same as in the
		results in section \ref{sec:stab_ambiguity}. The part of the constant stemming from phase
		retrieval is a slight modification from the constants in section \ref{sec:stab_ambiguity}.
		It is mostly the term $\lvert \phi(n_0)\phi(n_0 + \ell_\phi) \rvert$ in the denominator
		that deserves some attention. It is clear that phase propagation based on the relation
		\[
			\mathcal{F}^{-1}\left[ M_\phi[x](n,\cdot) \right](\ell_\phi)
				= x(n_0 + \ell_\phi + n)\overline{x(n_0 + n)}\phi(n_0)\overline{\phi(n_0+\ell_\phi)}
		\]
		will be unstable whenever the ends of the window $\phi(n_0)$ and $\phi(n_0 + \ell_\phi)$
		are close to zero. In particular, the reconstruction method proposed by the authors of
		\cite{Eldar15, Li17} benefits from windows for which $\lvert \phi(n_0)
		\phi(n_0 + \ell_\phi) \rvert$ is large such as the Hamming or rectangular window.
	\end{remark}
	As remarked before, in contrast to the stability result in \cite{Li17}, our result is
	applicable even when the Fourier transform of the magnitude squared of the window function
	$\mathcal{F}[\lvert \phi \rvert^2]$ has vanishing entries at the cost of only being applicable
	to bandlimited signals. We should also note that in \cite{Li17} the stability constant for the
	phase retrieval estimate scales like $\sqrt{\lvert \operatorname{supp} x \rvert} \cdot L^3$
	(which becomes $L^{7/2}$ for nowhere-vanishing signals) whereas our stability constant merely
	scales like $\sqrt{\lvert S_\delta \rvert L}$ (which becomes $L$ for signals whose entries have
	absolute values in excess of the threshold $\delta$).
	\begin{remark}[On disconnected graphs and duality results]
		Finally, we would like to note that one can prove a result resembling theorem
		\ref{thm:stabli_connected} in the case where $S_\delta$ has $K \in \mathbb{N}$
		connected components whose vertex sets are denoted by $S_1,\dots,S_K \subset S_\delta$.
		
		\vspace{5pt}
		Similarly, we may utilise lemma \ref{lem:autocorrelation} in order to deduce that
		\[
			\mathcal{F}\left[ M_\phi[x](\cdot,n) \right](m) = \frac{1}{\sqrt{L}} \cdot
			\sum_{k=0}^{L-1} \mathcal{F}[x](k) \overline{ \mathcal{F}[x](k-m) }
				\mathcal{F}[\phi](k-m-n) \overline{ \mathcal{F}[\phi](k-n) }
		\]
		holds, for $x,\phi \in \mathbb{C}^L$ and $m,n = 0,\dots,L-1$. This in turn can be used to
		deduce a stability result which is essentially the Fourier-dual of theorem
		\ref{thm:stabli_connected}.
	\end{remark}
	
	\section{Proofs of the main results}\label{sec:proofs}
	
	\begin{proof}[Proof of corollary \ref{cor:bojarovska_all}]
		By the ambiguity function relation and our assumptions, we find that 
		\[
			\mathcal{A}[x](m,n) = \mathcal{A}[y](m,n), \qquad \mbox{for } m \in \{0,\dots,\Delta+1\},
				~n\in\{0,\dots,L-1\}.
		\]
		Therefore, 
		\begin{equation}\label{eq:phase_prop}
			x(\ell) \overline{x(\ell-m)} = y(\ell) \overline{y(\ell-m)},
			\qquad \mbox{for } \ell \in \{0,\dots,L-1\},~m \in \{0,\dots,\Delta+1\},
		\end{equation}
		and in particular $x$ and $y$ have the same magnitudes. Let us now consider $k \in
		\{1,\dots,K\}$ as well as some $\ell_0$ in $V_k$. We have $\lvert x(\ell_0) \rvert =
		\lvert y(\ell_0) \rvert$ and hence $x(\ell_0) = \mathrm{e}^{\mathrm{i} \alpha_k}
		y(\ell_0)$, for some $\alpha_k \in \mathbb{R}$. As $V_k$ is the vertex set of a 
		connected component of $G$, it follows that for all $\ell \in V_k \setminus \{\ell_0\}$,
		there exists a (simple) path from $\ell_0$ to $\ell$. Therefore, we can consider $\ell \in
		V_k \setminus \{\ell_0\}$ and let $(u_0,\dots,u_n)$ be the vertex sequence of the path from
		$u_0 = \ell_0$ to $u_n = \ell$. For $j \in \{0,\dots,n-1\}$ one has, by definition of the
		edge set, that
		\[
			\lvert u_{j+1} - u_j \rvert \in (0,\Delta+1] \cup [L-\Delta-1,L).
		\]
		Thus, there exists an $m_j \in \{1,\dots,\Delta+1\}$ such that $u_{j+1} - u_j = m_j
		\mod L$ or $u_j - u_{j+1} = m_j \mod L$. In either case, it follows from equation
		\eqref{eq:phase_prop} that $x(u_{j+1}) \overline{x(u_j)} = y(u_{j+1}) \overline{y(u_j)}$.
		By induction on $j$, we find that $x(\ell) = \mathrm{e}^{\mathrm{i} \alpha_k} y(\ell)$.
	\end{proof}

	\begin{proof}[Proof of theorem \ref{thm:ambig_one_island}]
		The case $4B \leq L$ is similar to the case $4B > L$ but simpler: Indeed consider $4B \leq
		L$. In this case, we have 
		\[
			\min_{\substack{ m \in \{0,1\}\\n \in \{0,\dots,L-1\}}} \lvert \mathcal{A}[\phi](m,n) \rvert
			\geq \frac{1}{c}
		\]
		by assumption. Therefore, we can replace all sums over $\{-2B,\dots,2B\}$ by sums over
		$\{0,\dots,L-1\}$ in this proof. So let us consider $4B > L$:
		Let $\alpha \in \mathbb{R}$ be arbitrary. Employing proposition
		\ref{prop:split_phase_magnitude}, we have
		\begin{equation}\label{eq:NN}
			\left\lVert x - \mathrm{e}^{\mathrm{i} \alpha} y \right\rVert_{\ell^2(V_\delta)}
			\leq \left\lVert
				\lvert x \rvert - \lvert y \rvert
			\right\rVert_{\ell^2(V_\delta)} + \left( \sum_{\ell \in S_\delta}
				\lvert x(\ell) \rvert^2
				\left\lvert
					\frac{x(\ell)}{\lvert x(\ell) \rvert}
					- \mathrm{e}^{\mathrm{i} \alpha}
					\frac{y(\ell)}{\lvert y(\ell) \rvert} \right\rvert^2 \right)^\frac12.
		\end{equation}
		The magnitude difference is estimated as in lemma \ref{lem:magnitude_stab}. For the
		estimate of the phase difference, we develop inequalities in the following. Let $\ell,k
		\in S_\delta$. According to proposition \ref{prop:phase_propagation_estimate}, we have
		\begin{equation*}
			\left\lvert \frac{x(\ell)}{\lvert x(\ell) \rvert}
				- \mathrm{e}^{\mathrm{i} \alpha}
				\frac{y(\ell)}{\lvert y(\ell) \rvert} \right\rvert
			\leq \left\lvert \frac{x(k)}{\lvert x(k) \rvert}
				- \mathrm{e}^{\mathrm{i} \alpha}
				\frac{y(k)}{\vert y(k) \rvert} \right\rvert
			+ \frac{
				2 \left\lvert x(\ell) \overline{x(k)} - y(\ell) \overline{y(k)} \right\rvert
			}{
				\lvert x(\ell)x(k) \rvert
			}.
		\end{equation*}
		Using the above inequality recursively, one obtains that for all $M \in \mathbb{N}$ and
		$u_0,u_1,\dots,u_M \in S_\delta$:
		\begin{equation}\label{eq:phaseprop}
			\left\lvert \frac{x(u_M)}{\lvert x(u_M) \rvert} - \mathrm{e}^{\mathrm{i} \alpha}
				\frac{y(u_M)}{\lvert y(u_M) \rvert} \right\rvert 
			\leq \left\lvert \frac{x(u_0)}{\lvert x(u_0) \rvert} - \mathrm{e}^{\mathrm{i} \alpha}
				\frac{y(u_0)}{\lvert y(u_0) \rvert} \right\rvert
				+ 2 \sum_{j=0}^{M-1} \frac{
					\left\lvert 
					x(u_{j+1}) \overline{x(u_j)} - y(u_{j+1}) \overline{y(u_j)}
					\right\rvert
				}{\lvert x(u_{j+1})
				x(u_j) \rvert}.
		\end{equation}
		Suppose now that $\ell_0$ is chosen such
		that any other vertex $\ell \in S_\delta$ has graph distance (in the induced subgraph) at
		most $\lvert S_\delta \rvert /2 $ from $\ell_0$. Then, for any $\ell \in S_\delta \setminus
		\{ \ell_0 \}$, there exists $M(\ell) \in \mathbb{N}$, with $M(\ell) \leq \lvert S_\delta
		\rvert / 2$, and a sequence $u_0^\ell = \ell_0, u_1^\ell, \dots, u_{M(\ell)}^\ell = \ell$
		in $S_\delta$ such that (cf.~definition \ref{def:ess_supp} with $\Delta = 0$)
		\[
			\lvert u_{j+1}^\ell - u_j^\ell \rvert \in \{1,L-1\}, \qquad \mbox{for } j = 0,\dots,M(\ell)-1.
		\]
		Therefore, there exists a sequence $\sigma_1^\ell,\dots,\sigma_{M(\ell)}^\ell$ in $\{ -1,1
		\}$ such that
		\[
			u_{j+1}^\ell - u_j^\ell = \sigma_{j+1}^\ell \mod L, \qquad \mbox{for } j = 0,\dots,M(\ell)-1.
		\]
		Now, let $\alpha \in \mathbb{R}$ be such that 
		\[
			\left\lvert \frac{x(\ell_0)}{\lvert x(\ell_0) \rvert} -
			\mathrm{e}^{\mathrm{i} \alpha}
				\frac{y(\ell_0)}{\lvert y(\ell_0) \rvert} \right\rvert = 0.
		\]
		Then, we have for any $\ell \in S_\delta$, according to the above considerations (and inequality
		\eqref{eq:phaseprop}), that
		\begin{equation*}
			\left\lvert \frac{x(\ell)}{\lvert x(\ell) \rvert} - \mathrm{e}^{\mathrm{i} \alpha}
				\frac{y(\ell)}{\lvert y(\ell) \rvert} \right\rvert 
			\leq 2 \sum_{j=1}^{M(\ell)} \frac{
					\left\lvert 
					x(u_j^\ell) \overline{x(u_j^\ell - \sigma_j^\ell)} - y(u_j^\ell) \overline{y(u_j^\ell - \sigma_j^\ell)}
					\right\rvert
				}{\lvert x(u_j^\ell)
				x(u_j^\ell - \sigma_j^\ell) \rvert}.
		\end{equation*}
		For the second term of the right-hand side of inequality \eqref{eq:NN} this yields
		\begin{multline*}
			\left( \sum_{\ell \in S_\delta} \lvert x(\ell) \rvert^2 \left\lvert
				\frac{x(\ell)}{\lvert x(\ell) \rvert}
					- \mathrm{e}^{\mathrm{i} \alpha}
					\frac{y(\ell)}{\lvert y(\ell) \rvert}
			\right\rvert^2 \right)^\frac12 \\
			\leq 2 \left( \sum_{\ell \in S_\delta} \lvert x(\ell) \rvert^2 \left(
				\sum_{j=1}^{M(\ell)} \frac{
					\left\lvert 
						x(u_j^\ell) \overline{x(u_j^\ell - \sigma_j^\ell)}
							- y(u_j^\ell) \overline{y(u_j^\ell - \sigma_j^\ell)}
					\right\rvert
				}{
					\lvert x(u_j^\ell) x(u_j^\ell - \sigma_j^\ell) \rvert
				}
			\right)^2 \right)^\frac12.
		\end{multline*}
		Applying Jensen's inequality on the square of the inner sum and noting that $M(\ell) \leq
		\lvert S_\delta \rvert / 2$, we obtain
		\begin{multline*}
			\left( \sum_{\ell \in S_\delta} \lvert x(\ell) \rvert^2 \left\lvert
				\frac{x(\ell)}{\lvert x(\ell) \rvert}
					- \mathrm{e}^{\mathrm{i} \alpha}
					\frac{y(\ell)}{\lvert y(\ell) \rvert}
			\right\rvert^2 \right)^\frac12 \\
			\leq \sqrt{2 \lvert S_\delta \rvert} \left( \sum_{\ell \in S_\delta} \sum_{j=1}^{M(\ell)}
				\frac{
					\left\lvert x(\ell) \right\rvert^2
				}{
					\lvert x(u_j^\ell) x(u_j^\ell - \sigma_j^\ell) \rvert^2
				}
				\left\lvert 
					x(u_j^\ell) \overline{x(u_j^\ell - \sigma_j^\ell)}
						- y(u_j^\ell) \overline{y(u_j^\ell - \sigma_j^\ell)}
				\right\rvert^2
			\right)^\frac12.
		\end{multline*}
		Since $u^\ell_j \in S_\delta$, for $j \in \{1,\dots,M(\ell)\}$, we can further estimate
		\begin{multline*}
			\left( \sum_{\ell \in S_\delta} \lvert x(\ell) \rvert^2 \left\lvert
				\frac{x(\ell)}{\lvert x(\ell) \rvert}
					- \mathrm{e}^{\mathrm{i} \alpha}
					\frac{y(\ell)}{\lvert y(\ell) \rvert}
			\right\rvert^2 \right)^\frac12 \\
			\leq \frac{\sqrt{2 \lvert S_\delta \rvert}}{\delta^2}
			\left( \sum_{\ell \in S_\delta} \sum_{j=1}^{M(\ell)}
				\left\lvert x(\ell) \right\rvert^2
					\left\lvert 
						x(u_j^\ell) \overline{x(u_j^\ell - \sigma_j^\ell)}
							- y(u_j^\ell) \overline{y(u_j^\ell - \sigma_j^\ell)}
					\right\rvert^2
			\right)^\frac12
		\end{multline*}
		and with $\sigma^\ell_j \in \{ \pm 1 \}$, for $j \in \{1,\dots,M(\ell)\}$, we also get
		\begin{multline*}
			\left( \sum_{\ell \in S_\delta} \lvert x(\ell) \rvert^2 \left\lvert
				\frac{x(\ell)}{\lvert x(\ell) \rvert}
					- \mathrm{e}^{\mathrm{i} \alpha}
					\frac{y(\ell)}{\lvert y(\ell) \rvert}
			\right\rvert^2 \right)^\frac12 \\
			\leq \frac{\sqrt{2 \lvert S_\delta \rvert}}{\delta^2}
			\left( \sum_{\sigma \in \{-1,1\}} \sum_{\ell \in S_\delta} \sum_{j=1}^{M(\ell)}
				\left\lvert x(\ell) \right\rvert^2
					\left\lvert 
						x(u_j^\ell) \overline{x(u_j^\ell - \sigma)}
							- y(u_j^\ell) \overline{y(u_j^\ell - \sigma)}
					\right\rvert^2
			\right)^\frac12.
		\end{multline*}
		There are no repetitions in the sequences $u^\ell_1,u^\ell_2,\dots,u^\ell_{M(\ell)}$ and
		hence
		\begin{multline*}
			\left( \sum_{\ell \in S_\delta} \lvert x(\ell) \rvert^2 \left\lvert
				\frac{x(\ell)}{\lvert x(\ell) \rvert}
					- \mathrm{e}^{\mathrm{i} \alpha}
					\frac{y(\ell)}{\lvert y(\ell) \rvert}
			\right\rvert^2 \right)^\frac12 \\
			\leq \frac{\sqrt{2 \lvert S_\delta \rvert}}{\delta^2}
			\left( \sum_{\sigma \in \{-1,1\}} \sum_{\ell \in S_\delta} \sum_{u \in S_\delta}
				\left\lvert x(\ell) \right\rvert^2
					\left\lvert 
						x(u) \overline{x(u - \sigma)}
							- y(u) \overline{y(u - \sigma)}
					\right\rvert^2
			\right)^\frac12.
		\end{multline*}
		Therefore, we have 
		\begin{multline*}
			\left( \sum_{\ell \in S_\delta} \lvert x(\ell) \rvert^2 \left\lvert
				\frac{x(\ell)}{\lvert x(\ell) \rvert}
					- \mathrm{e}^{\mathrm{i} \alpha}
					\frac{y(\ell)}{\lvert y(\ell) \rvert}
			\right\rvert^2 \right)^\frac12 \\
			\leq \frac{
				\sqrt{2 \lvert S_\delta \rvert} \lVert x \rVert_{\ell^2(S_\delta)}
			}{
				\delta^2
			}
			\left( \sum_{\sigma \in \{-1,1\}} \sum_{u = 0}^{L-1}
				\left\lvert 
					x(u) \overline{x(u - \sigma)}
						- y(u) \overline{y(u - \sigma)}
				\right\rvert^2
			\right)^\frac12.
		\end{multline*}
		Suppose now that $4B < L$. By Plancherel's theorem, it holds that
		\begin{align*}
			&\left( \sum_{\ell \in S_\delta} \lvert x(\ell) \rvert^2 \left\lvert
				\frac{x(\ell)}{\lvert x(\ell) \rvert}
					- \mathrm{e}^{\mathrm{i} \alpha}
					\frac{y(\ell)}{\lvert y(\ell) \rvert}
			\right\rvert^2 \right)^\frac12 \\
			&\qquad\qquad\qquad\qquad\leq \frac{
				\sqrt{2 \lvert S_\delta \rvert} \lVert x \rVert_{\ell^2(S_\delta)}
			}{
				\delta^2
			}
			\left( \sum_{\sigma \in \{-1,1\}} \sum_{n = 0}^{L-1}
				\left\lvert 
					\mathcal{A}[x](\sigma,n)
						- \mathcal{A}[y](\sigma,n)
				\right\rvert^2
			\right)^\frac12 \\
			&\qquad\qquad\qquad\qquad= \frac{
				\sqrt{2 \lvert S_\delta \rvert} \lVert x \rVert_{\ell^2(S_\delta)}
			}{
				\delta^2
			}
			\left( \sum_{\sigma \in \{-1,1\}} \sum_{n = -2B}^{2B}
				\left\lvert 
					\mathcal{A}[x](\sigma,n)
						- \mathcal{A}[y](\sigma,n)
				\right\rvert^2
			\right)^\frac12.
		\end{align*}
		It follows from the ambiguity function relation and the lower bound on the ambiguity
		function of the window on $\{0,1\} \times \{-2B,\dots,2B\}$ that 
		\begin{align*}
			&\left( \sum_{\ell \in S_\delta} \lvert x(\ell) \rvert^2 \left\lvert
			\frac{x(\ell)}{\lvert x(\ell) \rvert}
				- \mathrm{e}^{\mathrm{i} \alpha}
				\frac{y(\ell)}{\lvert y(\ell) \rvert}
		\right\rvert^2 \right)^\frac12 \\
		&\qquad\qquad\qquad\qquad\leq c \cdot \frac{
			\sqrt{2 \lvert S_\delta \rvert} \lVert x \rVert_{\ell^2(S_\delta)}
		}{
			\delta^2
		}
		\left( \sum_{\sigma \in \{-1,1\}} \sum_{n = -2B}^{2B}
			\left\lvert
				\mathcal{F}\left[ M_\phi[x] - M_\phi[y] \right](n,-\sigma)
			\right\rvert^2
		\right)^\frac12 \\
		&\qquad\qquad\qquad\qquad\leq c \cdot \frac{
			\sqrt{2 \lvert S_\delta \rvert} \lVert x \rVert_{\ell^2(S_\delta)}
		}{
			\delta^2
		} \cdot \left\lVert \mathcal{F}\left[
			M_\phi[x] - M_\phi[y]
		\right] \right\rVert_{\mathrm{F}} \\
		&\qquad\qquad\qquad\qquad= c \cdot \frac{
			\sqrt{2 \lvert S_\delta \rvert} \lVert x \rVert_{\ell^2(S_\delta)}
		}{
			\delta^2
		} \cdot \left\lVert
			M_\phi[x] - M_\phi[y]
		\right\rVert_{\mathrm{F}},
	\end{align*}
	where we have used Plancherel's theorem in the last equality.
	\end{proof}
	
	\begin{proof}[Proof of theorem \ref{thm:fullstab_I}]
		Let $k \in \{1,\dots,K\}$ and $\alpha_k \in \mathbb{R}$. As in the proof of theorem
		\ref{thm:ambig_one_island}, we start by splitting the estimate into a phase and a magnitude
		estimate using proposition \ref{prop:split_phase_magnitude}:
		\[
			\left\lVert x - \mathrm{e}^{\mathrm{i} \alpha_k} y \right\rVert_{\ell^2(S_k)}
			\leq \left\lVert
				\lvert x \rvert - \lvert y \rvert
			\right\rVert_{\ell^2(S_k)} + \left( \sum_{\ell \in S_k}
				\lvert x(\ell) \rvert^2
				\left\lvert
					\frac{x(\ell)}{\lvert x(\ell) \rvert}
					- \mathrm{e}^{\mathrm{i} \alpha}
					\frac{y(\ell)}{\lvert y(\ell) \rvert} \right\rvert^2 \right)^\frac12.
		\]
		As the connected components $\{S_k\}_{k=1}^K$ are disjoint subsets of $V_\delta$, we can
		use Jensen's inequality to see that 
		\begin{align*}
			\sum_{k=1}^K \left\lVert
				\lvert x \rvert - \lvert y \rvert
			\right\rVert_{\ell^2(S_k)}
			\leq \sqrt{K} \left( \sum_{k=1}^K \left\lVert
				\lvert x \rvert - \lvert y \rvert
			\right\rVert_{\ell^2(S_k)}^2 \right)^\frac12
			&= \sqrt{K} \left\lVert
				\lvert x \rvert - \lvert y \rvert
			\right\rVert_{\ell^2(\bigcup_{k=1}^K S_k)} \\
			&\leq \sqrt{K} \left\lVert
				\lvert x \rvert - \lvert y \rvert
			\right\rVert_{\ell^2(V_\delta)}.
		\end{align*}
		Employing lemma \ref{lem:magnitude_stab}, we obtain for the magnitude retrieval estimate
		\[
			\left\lVert \lvert x \rvert - \lvert y \rvert \right\rVert_{\ell^2(V_\delta)} \leq
			\frac{c}{\delta} \left\lVert M_\phi[x]-M_\phi[y] \right\rVert_\mathrm{F}.
		\]
		The phase difference is estimated just like in theorem \ref{thm:ambig_one_island}: First, 
		observe that there must exist a vertex $\ell_0 \in S_k$ such
		that any other vertex $\ell \in S_k$ has graph distance $M(\ell)$ at most
		$\tfrac{L+\Delta}{2+\Delta}$ from $\ell_0$. Indeed, consider the following argument: The
		worst case which could happen is that we need to connect the vertex $0$ to the 
		vertex $\lfloor L/2 \rfloor$. By definition of the graph, it will take us exactly one step
		to go from $0$ to any $\ell \in \{1,\dots,\Delta+1\} \cap S_k$; it will take us exactly 
		two steps to go from $0$ to $\Delta + 2$, if the latter is in $S_k$; and it will take us at 
		most three steps to go from $0$ to $\ell \in \{\Delta+3,2\Delta+3\} \cap S_k$. Following 
		this logic, it is not too hard to see that it will take us at most $2n$ steps to go from 
		$0$ to $n(\Delta+2)$, if the latter is in $S_k$, and it will take us at most $2n+1$ 
		steps to go from $0$ to $\ell \in \{n(\Delta+2)+1,\dots,(n+1)(\Delta+2)-1\} \cap S_k$. So, 
		if there exists an element $n \in \mathbb{N}$ such that $n(\Delta+2) = \lfloor L/2 \rfloor$,
		then it will take us at most
		\[
			2n = \frac{2 \lfloor L/2 \rfloor}{\Delta + 2} \leq \frac{L}{\Delta+2}
			\leq \frac{L+\Delta}{2 + \Delta}
		\]
		steps to connect $0$ and $\lfloor L/2 \rfloor$. Similarly, if there is an element $n \in
		\mathbb{N}_0$ such that $\lfloor L/2 \rfloor \in \{n(\Delta+2)+1,\dots,
		(n+1)(\Delta+2)-1\}$, then there is a $\beta \in \{1,\dots,\Delta+1\}$ such that 
		\[
			2n+1 = \frac{2 \lfloor L/2 \rfloor - 2\beta}{\Delta + 2} + 1
			\leq \frac{L - 2}{\Delta + 2} + 1 = \frac{L + \Delta}{\Delta+2}.
		\]
		So for any $\ell \in S_k \setminus \{\ell_0\}$,
		there exists a path $u_0^\ell = \ell_0,u_1^\ell,\dots,u_{M(\ell)}^\ell = \ell$ from
		$\ell_0$ to $\ell$. By definition, this path satisfies 
		\[
			\left\lvert u_{j+1}^\ell - u_j^\ell \right\rvert \in (0,\Delta+1] \cup [L-\Delta-1,L),
			\qquad \mbox{for } j = 0,\dots,M(\ell)-1.
		\]
		Therefore, there exist sequences $\sigma_1^\ell,\dots,\sigma_{M(\ell)}^\ell \in \{\pm
		1\}$ and $\Delta_1^\ell,\dots,\Delta_{M(\ell)}^\ell \in \{1,\dots,\Delta+1\}$ such that 
		\[
			u_{j+1}^\ell - u_j^\ell = \sigma_{j+1}^\ell \Delta_{j+1}^\ell \mod L, \qquad \mbox{for } j = 0,
				\dots,M(\ell)-1.
		\]
		We let $\alpha_k \in \mathbb{R}$ be chosen in such a way that 
		\[
			\left\lvert \frac{x(\ell_0)}{\lvert x(\ell_0) \rvert} -
			\mathrm{e}^{\mathrm{i} \alpha_k}
				\frac{y(\ell_0)}{\lvert y(\ell_0) \rvert} \right\rvert = 0.
		\]
		Proceeding as in the proof of theorem \ref{thm:ambig_one_island} (now with $M(\ell) \leq
		(L+\Delta)/(2+\Delta)$), we derive
		\begin{multline*}
			\left( \sum_{\ell \in S_k} \lvert x(\ell) \rvert^2 \left\lvert
				\frac{x(\ell)}{\lvert x(\ell) \rvert}
					- \mathrm{e}^{\mathrm{i} \alpha_k}
					\frac{y(\ell)}{\lvert y(\ell) \rvert}
			\right\rvert^2 \right)^\frac12 \\
			\leq 2 \sqrt{\frac{L+\Delta}{2+\Delta}} \frac{1}{\delta^2}
			\left( \sum_{\ell \in S_k} \sum_{j=1}^{M(\ell)}
				\left\lvert x(\ell) \right\rvert^2
					\left\lvert 
						x(u_j^\ell) \overline{x(u_j^\ell - \sigma_j^\ell \Delta_j^\ell)}
							- y(u_j^\ell) \overline{y(u_j^\ell - \sigma_j^\ell \Delta_j^\ell)}
					\right\rvert^2
			\right)^\frac12.
		\end{multline*}
		We first treat the case $2\Delta < L-2$ and use that
		$\sigma_j^\ell \Delta_j^\ell \in \{-\Delta-1,\dots,\Delta+1\}$ to estimate
		\begin{multline*}
			\left( \sum_{\ell \in S_k} \lvert x(\ell) \rvert^2 \left\lvert
				\frac{x(\ell)}{\lvert x(\ell) \rvert}
					- \mathrm{e}^{\mathrm{i} \alpha_k}
					\frac{y(\ell)}{\lvert y(\ell) \rvert}
			\right\rvert^2 \right)^\frac12 \\
			\leq 2 \sqrt{\frac{L+\Delta}{2+\Delta}} \frac{1}{\delta^2}
			\left( \sum_{m = -\Delta-1}^{\Delta+1} \sum_{\ell \in S_k} \sum_{j=1}^{M(\ell)}
				\left\lvert x(\ell) \right\rvert^2
					\left\lvert 
						x(u_j^\ell) \overline{x(u_j^\ell - m)}
							- y(u_j^\ell) \overline{y(u_j^\ell - m)}
					\right\rvert^2
			\right)^\frac12.
		\end{multline*}
		We may use that there are no repetitions in the sequences $u_1^\ell,\dots,
		u_{M(\ell)}^\ell$ to obtain
		\begin{align*}
			&\left( \sum_{\ell \in S_k} \lvert x(\ell) \rvert^2 \left\lvert
				\frac{x(\ell)}{\lvert x(\ell) \rvert}
					- \mathrm{e}^{\mathrm{i} \alpha_k}
					\frac{y(\ell)}{\lvert y(\ell) \rvert}
			\right\rvert^2 \right)^\frac12 \\
			&\qquad\qquad\qquad\qquad\leq 2 \sqrt{\frac{L+\Delta}{2+\Delta}} \frac{1}{\delta^2}
			\left( \sum_{m = -\Delta-1}^{\Delta+1} \sum_{\ell \in S_k} \sum_{u = 0}^{L-1}
				\left\lvert x(\ell) \right\rvert^2
					\left\lvert 
						x(u) \overline{x(u - m)}
							- y(u) \overline{y(u - m)}
					\right\rvert^2
			\right)^\frac12 \\
			&\qquad\qquad\qquad\qquad= 2 \sqrt{\frac{L+\Delta}{2+\Delta}} \frac{\lVert x \rVert_{\ell^2(S_k)}}{\delta^2}
			\left( \sum_{m = -\Delta-1}^{\Delta+1} \sum_{u = 0}^{L-1}
				\left\lvert 
					x(u) \overline{x(u - m)}
						- y(u) \overline{y(u - m)}
				\right\rvert^2
			\right)^\frac12.
		\end{align*}
		Suppose furthermore that $4B > L$. According to Plancherel's theorem, we find that 
		\begin{align*}
			&\left( \sum_{\ell \in S_k} \lvert x(\ell) \rvert^2 \left\lvert
				\frac{x(\ell)}{\lvert x(\ell) \rvert}
					- \mathrm{e}^{\mathrm{i} \alpha_k}
					\frac{y(\ell)}{\lvert y(\ell) \rvert}
			\right\rvert^2 \right)^\frac12 \\
			&\qquad\qquad\qquad\qquad\leq 2 \sqrt{\frac{L+\Delta}{2+\Delta}} \frac{\lVert x \rVert_{\ell^2(S_k)}}{\delta^2}
			\left( \sum_{m = -\Delta-1}^{\Delta+1} \sum_{n = 0}^{L-1}
				\left\lvert 
					\mathcal{A}[x](m,n) - \mathcal{A}[y](m,n)
				\right\rvert^2
			\right)^\frac12 \\
			&\qquad\qquad\qquad\qquad= 2 \sqrt{\frac{L+\Delta}{2+\Delta}} \frac{\lVert x \rVert_{\ell^2(S_k)}}{\delta^2}
			\left( \sum_{m = -\Delta-1}^{\Delta+1} \sum_{n = -2B}^{2B}
				\left\lvert 
					\mathcal{A}[x](m,n) - \mathcal{A}[y](m,n)
				\right\rvert^2
			\right)^\frac12.
		\end{align*}
		Next, we use the ambiguity function relation, inequality \eqref{eq:low_bound_amb} as well
		as the symmetry of the ambiguity function of the window around the origin to derive
		\begin{align*}
			&\left( \sum_{\ell \in S_k} \lvert x(\ell) \rvert^2 \left\lvert
				\frac{x(\ell)}{\lvert x(\ell) \rvert}
					- \mathrm{e}^{\mathrm{i} \alpha_k}
					\frac{y(\ell)}{\lvert y(\ell) \rvert}
			\right\rvert^2 \right)^\frac12 \\
			&\qquad\qquad\qquad\qquad\leq 2c \sqrt{\frac{L+\Delta}{2+\Delta}} \frac{\lVert x \rVert_{\ell^2(S_k)}}{\delta^2}
			\left( \sum_{m = -\Delta-1}^{\Delta+1} \sum_{n = -2B}^{2B}
				\left\lvert 
					\mathcal{F}\left[ M_\phi[x] - M_\phi[y] \right](n,m)
				\right\rvert^2
			\right)^\frac12 \\
			&\qquad\qquad\qquad\qquad\leq 2c \sqrt{\frac{L+\Delta}{2+\Delta}} \frac{\lVert x \rVert_{\ell^2(S_k)}}{\delta^2}
			\left\lVert \mathcal{F}\left[ M_\phi[x] - M_\phi[y] \right] \right\rVert_2 \\
			&\qquad\qquad\qquad\qquad= 2c \sqrt{\frac{L+\Delta}{2+\Delta}} \frac{\lVert x \rVert_{\ell^2(S_k)}}{\delta^2}
			\left\lVert M_\phi[x] - M_\phi[y] \right\rVert_2.
		\end{align*}
		Note that we may once again use Jensen's inequality to show that 
		\[
			\sum_{k=1}^K \lVert x \rVert_{\ell^2(S_k)} \leq \sqrt{K} \lVert x \rVert_2.
		\]
		Thus, combining the phase and the magnitude estimates yields
		\[
			\inf_{\alpha_1,\dots,\alpha_K} \sum_{k=1}^K \left\lVert x - \mathrm{e}^{\mathrm{i} \alpha_k}
				y \right\rVert_{\ell^2(S_k)} \leq \frac{c \sqrt{K}}{\delta}
				\left( 1 + 2 \sqrt{\frac{L+\Delta}{2+\Delta}} \frac{\lVert x \rVert_2}{\delta} \right)
				\left\lVert M_\phi[x] - M_\phi[y] \right\rVert_2.
		\]
		The cases in which $2\Delta \geq L-2$ or $4B \geq L$ are dealt with similarly.
	\end{proof}
	
	\begin{proof}[Proof of theorem \ref{thm:stabli_connected}]
		Let $\alpha \in \mathbb{R}$. As in the proof of theorem \ref{thm:ambig_one_island}, we
		start by splitting the estimate into a phase and a magnitude estimate
		\[
			\left\lVert x - \mathrm{e}^{\mathrm{i} \alpha_k} y \right\rVert_{\ell^2(V_\delta)}
			\leq \left\lVert
				\lvert x \rvert - \lvert y \rvert
			\right\rVert_{\ell^2(V_\delta)} + \left( \sum_{\ell \in S_\delta}
				\lvert x(\ell) \rvert^2
				\left\lvert
					\frac{x(\ell)}{\lvert x(\ell) \rvert}
					- \mathrm{e}^{\mathrm{i} \alpha}
					\frac{y(\ell)}{\lvert y(\ell) \rvert} \right\rvert^2 \right)^\frac12.
		\]
		First, noting that
		\[
			\mathcal{F}\left[ \lvert \phi \rvert^2 \right](k) = \mathcal{A}[\phi](0,k),
			\qquad \mbox{for } k = 0,\dots,L-1,
		\]
		we can apply lemma \ref{lem:magnitude_stab} to obtain the estimate 
		\[
			\left\lVert \lvert x \rvert - \lvert y \rvert \right\rVert_{\ell^2(V_\delta)} \leq
			\frac{c}{\delta} \left\lVert M_\phi[x]-M_\phi[y] \right\rVert_\mathrm{F}.
		\]
		For the estimate on the phase retrieval part, we need to consider a new strategy based on
		equation \eqref{eq:autocorrelation}. First, we find that there must exist a vertex $\ell_0
		\in S_\delta$ such that any other vertex $\ell \in S_\delta$ has graph distance $M(\ell)$
		at most $\lvert S_\delta \rvert/2$ from $\ell_0$. So for any $\ell \in S_\delta \setminus
		\{\ell_0\}$, there exists a path $u_0^\ell = \ell_0,u_1^\ell,\dots,u_{M(\ell)}^\ell = \ell$
		from $\ell_0$ to $\ell$. By definition, this path satisfies 
		\[
			\left\lvert u_{j+1}^\ell - u_j^\ell \right\rvert \in \{\ell_\phi, L-\ell_\phi\},
			\qquad \mbox{for } j = 0,\dots,M(\ell)-1.
		\]
		Therefore, there exists a sequence $\sigma_1^\ell,\dots,\sigma_{M(\ell)}^\ell \in \{\pm
		1\}$ such that 
		\[
			u_{j+1}^\ell - u_j^\ell = \sigma_{j+1}^\ell \ell_\phi \mod L, \qquad \mbox{for } j = 0,
				\dots,M(\ell)-1.
		\]
		With this at hand, we proceed similarly to the proof of theorem \ref{thm:ambig_one_island}.
		We let $\alpha \in \mathbb{R}$ be chosen in such a way that 
		\[
			\left\lvert \frac{x(\ell_0)}{\lvert x(\ell_0) \rvert} -
			\mathrm{e}^{\mathrm{i} \alpha}
				\frac{y(\ell_0)}{\lvert y(\ell_0) \rvert} \right\rvert = 0.
		\]
		Then, we have that for any $\ell \in S_\delta$
		\begin{equation*}
			\left\lvert \frac{x(\ell)}{\lvert x(\ell) \rvert} - \mathrm{e}^{\mathrm{i} \alpha}
				\frac{y(\ell)}{\lvert y(\ell) \rvert} \right\rvert 
			\leq 2 \sum_{j=1}^{M(\ell)} \frac{
					\left\lvert 
					x(u_j^\ell) \overline{x(u_j^\ell - \sigma_j^\ell \ell_\phi)}
						- y(u_j^\ell) \overline{y(u_j^\ell - \sigma_j^\ell \ell_\phi)}
					\right\rvert
				}{\lvert x(u_j^\ell)
				x(u_j^\ell - \sigma_j^\ell \ell_\phi) \rvert}
		\end{equation*}
		and
		\begin{multline*}
			\left( \sum_{\ell \in S_\delta} \lvert x(\ell) \rvert^2 \left\lvert
				\frac{x(\ell)}{\lvert x(\ell) \rvert}
					- \mathrm{e}^{\mathrm{i} \alpha}
					\frac{y(\ell)}{\lvert y(\ell) \rvert}
			\right\rvert^2 \right)^\frac12 \\
			\leq \frac{\sqrt{2 \lvert S_\delta \rvert}}{\delta^2} \left( \sum_{\ell \in S_\delta} \sum_{j=1}^{M(\ell)}
			\left\lvert x(\ell) \right\rvert^2
			\left\lvert 
				x(u_j^\ell) \overline{x(u_j^\ell - \sigma_j^\ell \ell_\phi)}
					- y(u_j^\ell) \overline{y(u_j^\ell - \sigma_j^\ell \ell_\phi)}
				\right\rvert^2
			\right)^\frac12,
		\end{multline*}
		by Jensen's inequality, $M(\ell) \leq \lvert S_\delta \rvert / 2$ and the fact that
		$u_j^\ell, u_j^\ell - \sigma_j^\ell \ell_\phi \in S_\delta$, for all $j \in
		\{1,\dots,M(\ell)\}$. As before, we can further employ a crude estimate to derive
		\begin{multline*}
			\left( \sum_{\ell \in S_\delta} \lvert x(\ell) \rvert^2 \left\lvert
				\frac{x(\ell)}{\lvert x(\ell) \rvert}
					- \mathrm{e}^{\mathrm{i} \alpha}
					\frac{y(\ell)}{\lvert y(\ell) \rvert}
			\right\rvert^2 \right)^\frac12 \\
			= \frac{\sqrt{2 \lvert S_\delta \rvert} \lVert x \rVert_{\ell^2(S_\delta)}}{\delta^2} \left(
			\sum_{\sigma \in \{\pm 1\}} \sum_{u \in S_\delta}
				\left\lvert 
					x(u) \overline{x(u - \sigma \ell_\phi)}
						- y(u) \overline{y(u - \sigma \ell_\phi)}
					\right\rvert^2
				\right)^\frac12,
		\end{multline*}
		since $\sigma_j^\ell \in \{\pm 1\}$, for all $j \in \{1,\dots,M(\ell)\}$ and all $\ell \in
		S_\delta$, and because for fixed $\ell \in S_\delta$, the $u_j^\ell$ are all distinct.
		Next, we note that due to $\operatorname{supp} \phi = \{n_0,\dots,n_0 + \ell_\phi\}
		\mod L$, $2\ell_\phi < L$, and the autocorrelation relation, we have for $u \in S_\delta$
		\begin{align*}
			x(u)\overline{x(u-\ell_\phi)} \phi(n_0) \overline{\phi(n_0 + \ell_\phi)}
			&= \sum_{\ell = 0}^{L-1} x(\ell) \overline{x(\ell-\ell_\phi)} 
				\phi(\ell + n_0 - u) \overline{ \phi(\ell + n_0 + \ell_\phi - u) } \\
			&= \sqrt{L} \mathcal{F}^{-1} [M_\phi[x](u-n_0-\ell_\phi,\cdot)](\ell_\phi),
		\end{align*}
		as well as 
		\begin{align*}
			x(u)\overline{x(u+\ell_\phi)} \phi(n_0+\ell_\phi) \overline{\phi(n_0)}
			&= \sum_{\ell = 0}^{L-1} x(\ell) \overline{x(\ell+\ell_\phi)} 
				\phi(\ell + \ell_\phi + n_0 - u) \overline{ \phi(\ell + n_0 - u) } \\
			&= \sqrt{L} \mathcal{F}^{-1} [M_\phi[x](u-n_0,\cdot)](-\ell_\phi).
		\end{align*}
		This implies
		\begin{multline*}
			\left( \sum_{\ell \in S_\delta} \lvert x(\ell) \rvert^2 \left\lvert
				\frac{x(\ell)}{\lvert x(\ell) \rvert}
					- \mathrm{e}^{\mathrm{i} \alpha}
					\frac{y(\ell)}{\lvert y(\ell) \rvert}
			\right\rvert^2 \right)^\frac12 \\
			\leq \frac{\sqrt{2 \lvert S_\delta \rvert L} \lVert x \rVert_{\ell^2(S_\delta)}}{\delta^2
			\lvert \phi(n_0) \phi(n_0 + \ell_\phi) \rvert} \Bigg(
				\sum_{u \in S_\delta}
					\Big( \left\lvert 
						\mathcal{F}^{-1} [M_\phi[x](u-n_0-\ell_\phi,\cdot) - M_\phi[y](u-n_0-\ell_\phi,\cdot)](\ell_\phi)
					\right\rvert^2 \\
					+ \left\lvert 
						\mathcal{F}^{-1} [M_\phi[x](u-n_0,\cdot) - M_\phi[y](u-n_0,\cdot)](-\ell_\phi)
					\right\rvert^2 \Big) \Bigg)^\frac12.
		\end{multline*}
		Yet another crude estimate results in
		\begin{align*}
			&\left( \sum_{\ell \in S_\delta} \lvert x(\ell) \rvert^2 \left\lvert
				\frac{x(\ell)}{\lvert x(\ell) \rvert}
					- \mathrm{e}^{\mathrm{i} \alpha}
					\frac{y(\ell)}{\lvert y(\ell) \rvert}
			\right\rvert^2 \right)^\frac12 \\
			&\qquad\qquad\qquad\qquad\leq \frac{2 \sqrt{\lvert S_\delta \rvert L} \lVert x \rVert_{\ell^2(S_\delta)}}{\delta^2
			\lvert \phi(n_0) \phi(n_0 + \ell_\phi) \rvert} \left(
				\sum_{m,n = 0}^{L-1}
					\left\lvert 
						\mathcal{F}^{-1} [M_\phi[x](m,\cdot) - M_\phi[y](m,\cdot)](n)
					\right\rvert^2 \right)^\frac12 \\
			&\qquad\qquad\qquad\qquad= \frac{2 \sqrt{\lvert S_\delta \rvert L} \lVert x \rVert_{\ell^2(S_\delta)}}{\delta^2
			\lvert \phi(n_0) \phi(n_0 + \ell_\phi) \rvert} \left\lVert M_\phi[x] - M_\phi[y] \right\rVert_{\mathrm{F}}.
		\end{align*}
	\end{proof}
	
	%----------------------------- Hidden Appendix -----------------------------
	\appendix
	\section{Pointwise estimates for the stability results}
	Let us start by the typical splitting of signal differences into phase and magnitude part.
	\begin{proposition}\label{prop:split_phase_magnitude}
		Let $x,y \in \mathbb{C}^L$. Then, we have for all
		$\alpha \in \mathbb{R}$ that
		\[
			\lvert x(\ell) - \mathrm{e}^{\mathrm{i} \alpha} y(\ell) \rvert
			\leq \left\lvert \lvert x(\ell) \rvert - \lvert y(\ell) \rvert \right\rvert
			+ \lvert x(\ell) \rvert \left\lvert \frac{x(\ell)}{\lvert x(\ell) \rvert} - \mathrm{e}^{\mathrm{i} \alpha} \frac{y(\ell)}{\lvert y(\ell) \rvert} \right\rvert
		\]
		holds for all $\ell \in \{0,\dots,L-1\}$ for which $x(\ell) \neq 0$ and $y(\ell) \neq 0$. Moreover,
		\[
			\lvert x(\ell) - \mathrm{e}^{\mathrm{i} \alpha} y(\ell) \rvert
			= \left\lvert \lvert x(\ell) \rvert - \lvert y(\ell) \rvert \right\rvert
		\]
		holds for all $\ell \in \{0,\dots,L-1\}$ for which $x(\ell) = 0$ or $y(\ell) = 0$.
	\end{proposition}
	\begin{proof}
		Consider the case $x(\ell) \neq 0$ and $y(\ell) \neq 0$:
		\begin{align*}
			\lvert x(\ell) - \mathrm{e}^{\mathrm{i} \alpha} y(\ell) \rvert
			&= \left\lvert \lvert x(\ell) \rvert \frac{x(\ell)}{\lvert x(\ell) \rvert}
				- \mathrm{e}^{\mathrm{i} \alpha} \lvert y(\ell) \rvert
					\frac{y(\ell)}{\lvert y(\ell) \rvert} \right\rvert \\
			&= \left\lvert \left(
				\lvert x(\ell) \rvert - \lvert y(\ell) \rvert
			\right) \frac{\mathrm{e}^{\mathrm{i} \alpha} y(\ell)}{\lvert y(\ell) \rvert}
				+ \lvert x(\ell) \rvert \left( \frac{x(\ell)}{\lvert x(\ell) \rvert}
					- \mathrm{e}^{\mathrm{i} \alpha} \frac{y(\ell)}{\lvert y(\ell) \rvert} \right) \right\rvert \\
			&\leq \left\lvert \lvert x(\ell) \rvert - \lvert y(\ell) \rvert \right\rvert
				+ \lvert x(\ell) \rvert \left\lvert
					\frac{x(\ell)}{\lvert x(\ell) \rvert} 
					- \mathrm{e}^{\mathrm{i} \alpha} \frac{y(\ell)}{\lvert y(\ell) \rvert}
				\right\rvert.
		\end{align*}
		If $x(\ell) = 0$, then
		\[
		\lvert x(\ell) - \mathrm{e}^{\mathrm{i} \alpha} y(\ell) \rvert = \lvert y(\ell) \rvert = \left\lvert \lvert x(\ell) \rvert - \lvert y(\ell) \rvert \right\rvert.
		\]
	\end{proof}
	In addition, a result about phase propagation will be handy.
	\begin{proposition}\label{prop:phase_propagation_estimate}
		Let $x,y \in \mathbb{C}^L$ and let $\alpha \in \mathbb{R}$.
		Then, we have that
		\[
			\left\lvert \frac{x(k)}{\lvert x(k) \rvert}
				- \mathrm{e}^{\mathrm{i} \alpha} \frac{y(k)}{\lvert y(k) \rvert} \right\rvert
			\leq \left\lvert \frac{x(\ell)}{\lvert x(\ell) \rvert}
				- \mathrm{e}^{\mathrm{i} \alpha} \frac{y(\ell)}{\lvert y(\ell) \rvert} \right\rvert
			+ \frac{
				2 \left\lvert x(k) \overline{x(\ell)}- y(k) \overline{y(\ell)} \right\rvert
			}{
				\lvert x(k)x(\ell) \rvert
			}
		\]
		holds, for all $\ell, k \in \{0,\dots,L-1\}$ for which $x(k), x(\ell), y(k)$ and $y(\ell)$
		are all different from zero.
	\end{proposition}
	\begin{proof}
		We compute
		\begin{align*}
			\left\lvert \frac{x(k)}{\lvert x(k) \rvert} - \mathrm{e}^{\mathrm{i} \alpha} \frac{y(k)}{\lvert y(k) \rvert} \right\rvert &= \left\lvert \frac{x(k)\overline{x(\ell)} x(\ell)}{\lvert x(k) \rvert\lvert x(\ell) \rvert^2} - \mathrm{e}^{\mathrm{i} \alpha} \frac{y(k)\overline{y(\ell)} y(\ell)}{\lvert y(k) \rvert\lvert y(\ell) \rvert^2} \right\rvert \\
			&= \left\lvert \left( \frac{x(\ell)}{\lvert x(\ell) \rvert} - \mathrm{e}^{\mathrm{i} \alpha} \frac{y(\ell)}{\lvert y(\ell) \rvert} \right) \frac{x(k)\overline{x(\ell)}}{\lvert x(k)x(\ell) \rvert} + \mathrm{e}^{\mathrm{i} \alpha} \frac{y(\ell)}{\lvert y(\ell) \rvert}\left( \frac{x(k)\overline{x(\ell)}}{\lvert x(k)x(\ell) \rvert} - \frac{y(k)\overline{y(\ell)}}{\lvert y(k)y(\ell) \rvert} \right) \right\rvert \\
			&\leq \left\lvert \frac{x(\ell)}{\lvert x(\ell) \rvert} - \mathrm{e}^{\mathrm{i} \alpha} \frac{y(\ell)}{\lvert y(\ell) \rvert} \right\rvert + \left\lvert \frac{x(k)\overline{x(\ell)}}{\lvert x(k)x(\ell) \rvert} - \frac{y(k)\overline{y(\ell)}}{\lvert y(k)y(\ell) \rvert} \right\rvert.
		\end{align*}
		The claim follows by noting that using the triangle inequality and the reverse triangle
		inequality, one has
		\[
			\left\lvert \frac{z_0}{\lvert z_0 \rvert} - \frac{z_1}{\lvert z_1 \rvert} \right\rvert
			= \frac{\big\lvert z_0 \lvert z_1 \rvert - \lvert z_0 \rvert z_1 \big\rvert}{\lvert z_0z_1 \rvert}
			\leq \frac{
				\lvert z_0 - z_1 \rvert + \big\lvert \lvert z_0 \rvert-\lvert z_1 \rvert \big\rvert
			}{
				\lvert z_0 \rvert
			} \leq \frac{2\lvert z_0 - z_1 \rvert}{\lvert z_0 \rvert},
		\]
		for $z_0,z_1 \in \mathbb{C}$.
	\end{proof}

	\section{Proofs of the most fundamental formulae}\label{app:proofs_conseq_formulae}
	
	For convenience of the reader, we present the proofs of the two formulae presented in
	section \ref{sec:prerequisites}. We note that these formulae are well-known in the
	literature and have repeatedly been used to prove uniqueness results in recent years
	\cite{Bojarovska15, Eldar15, Li17, Nawab83}. We start with lemma \ref{lem:autocorrelation}.
	
	\begin{proof}[Proof of lemma \ref{lem:autocorrelation}]
		For $m,n \in \{0,\dots,L-1\}$, one can calculate
		\begin{align*}
			M_\phi[x](m,n) &= \left\lvert \mathcal{F}[x_m](n) \right\rvert^2 
			= \mathcal{F}[x_m](n) \cdot \overline{\mathcal{F}[x_m](n)} = \mathcal{F}[x_m](n) \cdot \mathcal{F}[x_m^\#](n) \\
			&= \frac{1}{\sqrt{L}} \cdot \mathcal{F}[x_m \ast x_m^\#](n),
		\end{align*}
		using the convolution theorem for the DFT. Applying the inverse Fourier transform 
		yields the statement.
	\end{proof}
	
	\begin{proof}[Proof of lemma \ref{lem:ambguity}]
		According to lemma \ref{lem:autocorrelation}, we have
		\[
			\mathcal{F}^{-1}\left[ M_\phi[x](m,\cdot) \right](n)
			= \frac{1}{\sqrt{L}} \cdot \sum_{\ell=0}^{L-1} 
				x(\ell)\overline{x(\ell-n)} \phi(\ell-n-m) \overline{\phi(\ell-m)},
		\]
		for $m,n \in \{0,\dots,L-1\}$, and therefore
		\[
			\mathcal{F}\left[ M_\phi[x](m,\cdot) \right](n) 
			= \frac{1}{\sqrt{L}} \cdot \sum_{\ell=0}^{L-1} 
				x(\ell)\overline{x(\ell+n)} \phi(\ell+n-m) \overline{\phi(\ell-m)}.
		\]
		Taking the DFT in $m$ yields
		\begin{align*}
			\mathcal{F}\left[ M_\phi[x]\right](m,n) 
			&= \frac{1}{L} \cdot \sum_{\ell,k=0}^{L-1} 
				x(\ell)\overline{x(\ell+n)} \phi(\ell+n-k) \overline{\phi(\ell-k)} 
					\mathrm{e}^{-2\pi\mathrm{i}\frac{k m}{L}} \\
			&= \frac{1}{L} \cdot \sum_{\ell=0}^{L-1}
				x(\ell)\overline{x(\ell+n)} \mathrm{e}^{-2\pi\mathrm{i}\frac{\ell m}{L}}
				\sum_{k=0}^{L-1} \phi(\ell+n-k) \overline{\phi(\ell-k)}
					\mathrm{e}^{2\pi\mathrm{i}\frac{(\ell-k) m}{L}} \\
			&= \mathcal{A}[x](-n,m) \overline{\mathcal{A}[\phi](-n,m)}.
		\end{align*}
	\end{proof}

	\begin{proof}[Proof of proposition \ref{prop:bandlimited_amb}]
		Let us use the notation $x'_m(\ell) = \overline{x(\ell-m)}$, for $\ell,m
		\in \{0,\dots,L-1\}$, and denote the entry-wise product of two vectors via $\odot$.
		For $m,n \in \{0,\dots,L-1\}$ we apply the convolution theorem for the DFT to 
		see that 
		\[
			\mathcal{A}[x](m,n) = \mathcal{F}[x \odot x'_m](n)
				= \frac{1}{\sqrt{L}} \left( \mathcal{F}[x] \ast \mathcal{F}[x'_m] \right)(n).
		\]
		Next, we compute 
		\[
			\mathcal{F}[x'_m](k) = \frac{1}{\sqrt{L}} \sum_{\ell = 0}^{L-1}
				\overline{x(\ell - m)} \mathrm{e}^{-2\pi\mathrm{i} \frac{\ell k}{L}}
			= \frac{\mathrm{e}^{-2\pi\mathrm{i} \frac{m k}{L}}}{\sqrt{L}} \sum_{\ell = 0}^{L-1}
				\overline{x(\ell)} \mathrm{e}^{-2\pi\mathrm{i} \frac{\ell k}{L}}
			= \mathrm{e}^{-2\pi\mathrm{i} \frac{m k}{L}} \mathcal{F}[x]^\#(k),
		\]
		for $k \in \{0,\dots,L-1\}$. Therefore, 
		\begin{align*}
			\mathcal{A}[x](m,n)
			&= \frac{1}{\sqrt{L}} \left( \mathcal{F}[x] \ast \mathcal{F}[x'_m] \right)(n)
			= \frac{1}{\sqrt{L}} \sum_{k = 0}^{L-1} \mathcal{F}[x](k) \overline{\mathcal{F}[x](k-n)}
				\mathrm{e}^{2\pi\mathrm{i} \frac{m (k-n)}{L}} \\
			&= \frac{1}{\sqrt{L}} \sum_{k = -B}^B \mathcal{F}[x](k) \overline{\mathcal{F}[x](k-n)}
				\mathrm{e}^{2\pi\mathrm{i} \frac{m (k-n)}{L}}.
		\end{align*}
		Let us assume that $4B < L$ and consider $k \in \{-B,\dots,B\}$. First, suppose that
		$n \in (-\tfrac{L}{2},-2B) \cap \mathbb{Z}$. Then, it follows that 
		\[ 
			k - n \in [-B-n,B-n] \cap \mathbb{Z} \subset \left(B,\frac{L}{2} + B \right) \cap \mathbb{Z}
		\]
		and since $L-B \geq \tfrac{L}{2} + B$, it follows that $\mathcal{F}[x](k-n) = 0$. Therefore, 
		\[
			\mathcal{A}[x](m,n) = 0.
		\]
		Secondly, suppose that $n \in (2B,\tfrac{L}{2}] \cap \mathbb{Z}$. Then, one has
		\[ 
			k - n \in [-B-n,B-n] \cap \mathbb{Z} \subset \left[-\frac{L}{2} - B,- B \right) \cap \mathbb{Z}
		\]
		and $\tfrac{L}{2} - B > B$ implies $\mathcal{F}[x](k-n) = 0$. Thus, 
		\[
			\mathcal{A}[x](m,n) = 0.
		\] 
		What remains is the case $4B \geq L$: Note that by $\{-2B,\dots,2B\}
		= \{0,\dots,L-1\} \mod L$ the proposition is trivial in this case.
	\end{proof}

	%------------------------------ Bibliography -------------------------------
\bibliographystyle{plain}
\bibliography{sources}
\end{document}